\documentclass[11pt,a4paper,reqno]{amsart}
\usepackage{amsmath,amsthm,amsfonts,amssymb,bm,wasysym}
\usepackage{caption}
\usepackage{subcaption}
\usepackage{fullpage}
\usepackage{graphicx}
\usepackage{tikz,ifthen}
\usepackage{hyperref}
\usetikzlibrary{calc}
\usetikzlibrary{patterns}
\usetikzlibrary{arrows}
\usetikzlibrary{decorations.pathreplacing}
\theoremstyle{plain}
\newtheorem{theorem}{Theorem}
\newtheorem{proposition}[theorem]{Proposition}

\newtheorem{lemma}[theorem]{Lemma}

\newtheorem{conjecture}[theorem]{Conjecture}

\theoremstyle{definition}

\theoremstyle{remark}

\begin{document}

\title{From heavy-tailed Boolean models to scale-free Gilbert graphs}
\def\MLine#1{\par\hspace*{-\leftmargin}\parbox{\textwidth}{\[#1\]}}

\def\A{\mathbb{A}}
\def\Ab{\mathcal{A}b}
\def\absq{{a^{\prime}}^2+{b^\prime}^2}
\def\AP{\text{G}}
\def\app{{a^{\prime\prime}}^2+1}
\def\argmin{\text{argmin}}
\def\arb{arbitrary }
\def\ass{assumption}
\def\arrow{\rightarrow}
\def\BT{B^{\mathbb{T}_n}}
\def\codim{\text{codim}}
\def\const{c}
\def\CCG{\text{G}}
\def\colim{\text{colim}}
\def\cond{condition }
\def\C{\mbox{\bf C}}
\def\ct{\mathsf{ct}}
\def\d{{\rm d}}
\def\dell{\partial}
\def\diam{\text{diam}}
\def\E{\mathbb{E}}
\def\envi{\mathsf{env}}
\def\enviIn{\partial^{\mathsf{in}}}
\def\enviOut{\partial^{\mathsf{out}}}
\def\enviInn{\partial^{\mathsf{in},*}}
\def\enviStab{\mathsf{env}_{\mathsf{stab}}}
\def\Et{\text{Et}}
\def\es{\emptyset}
\def\exp{\text{exp}}
\def\fa{for all }
\def\Fk{\mathcal{F}_{k_0}}
\def\Fm{Furthermore}
\def\G{\mathbb{G}}
\def\gr{\text{gr}}
\def\hge{h_{\mathsf{g}}}
\def\hco{h_{\mathsf{c}}}
\def\H{\text{H}}
\def\Hom{\text{Hom}}
\def\Hs{\widetilde{X}_{H,0}}
\def\inj{\hookrightarrow}
\def\id{\text{id}}
\def\iiets{it is easy to see }
\def\iietc{it is easy to check }
\def\Iietc{It is easy to check }
\def\Iiets{It is easy to see }
\def\imp{\Rightarrow}
\def\({\big(}
\def\){\big)}
\def\lver{\big|}
\def\rver{\big|}
\def\lcu{\big\{}
\def\rcu{\big\}}
\def\im{\mbox{im}}
\def\inn{\mathsf{in}}
\def\Inv{\text{Inv}}
\def\Ind{\text{Ind}}
\def\Ip{In particular}
\def\ip{in particular }
\def\LB{\text{LB}}
\def\Lo{\mathcal{L}^o}
\def\mc{\mathcal}
\def\mb{\mathbb}
\def\mf{\mathbf}
\def\Hp{\wt{X}_{H,0}^{'}}
\def\M{\mathbb{M}}
\def\Mo{Moreover}
\def\G{\mathbb{G}}
\def\GR{\mathsf{GR}}
\def\N{\mathbb{N}}
\def\Npo{\mathbf{N}_{\mathcal{P}^o}}
\def\k{\overline{k}}
\def\K{\underline{K}}
\def\LE{\mathsf{LE}}
\def\ldot{.}
\def\lmid{\;\middle\vert\;}
\def\O{\mathcal{O}}
\def\Ob{Observe }
\def\ob{observe }
\def\out{\mathsf{out}}
\def\Otoh{On the other hand}
\def\opartial{\partial^{\text{out}}}
\def\ipartial{\partial^{\text{in}}}
\def\eopartial{\partial^{\text{out}}_{\text{ext}}}
\def\eipartial{\partial^{\text{in}}_{\text{ext}}}
\def\p{\prime}
\def\pred{\mathsf{pred}}
\def\Trace{\mathsf{Trace}}
\def\pp{{\prime\prime}}
\def\Po{\mathcal{P}^0}
\def\P{\mathbb{P}}
\def\Proj{\mbox{\bf P}}
\def\Q{\mathbb{Q}}
\def\QQ{\overline{\Q}}
\def\pr{\text{pr}}
\def\R{\mathbb{R}}
\def\rstab{R_{\mathsf{stab}}}
\def\Spec{\text{Spec}}
\def\st{such that }
\def\sl{sufficiently large }
\def\ss{sufficiently small }
\def\sot{so that }
\def\su{suppose }
\def\succ{\mathsf{succ}}
\def\Su{Suppose }
\def\suf{sufficiently }
\def\udot{\mathaccent\cdot\cup}
\def\Set{\mathcal{S}et}
\def\T{\mathbb{T}}
\def\Twh{Then we have }
\def\Tes{There exists }
\def\te{there exist }
\def\tes{there exists }
\def\tptc{this proves the claim}
\def\Map{\text{Map}}
\def\VLo{\mc{VL}^o}
\def\wt{\widetilde}
\def\Wcon{We conclude }
\def\wcon{we conclude }
\def\wc{we compute }
\def\Wc{We compute }
\def\wo{we obtain }
\def\wh{we have }
\def\Wh{We have }
\def\Z{\mathbb{Z}}
\def\ZSlab{\mathbb{Z}^2_L\times\{0\}^{d-2}}

\author{Christian Hirsch}
\thanks{Weierstrass Institute Berlin, Mohrenstr. 39, 10117 Berlin, Germany; E-mail: {\tt hirsch@wias-berlin.de}.}

\begin{abstract}
Define the scale-free Gilbert graph based on a Boolean model with heavy-tailed radius distribution on the $d$-dimensional torus by connecting two centers of balls by an edge if at least one of the balls contains the center of the other. We investigate two asymptotic properties of this graph as the size of the torus tends to infinity. First, we determine the tail index associated with the asymptotic distribution of the sum of all power-weighted incoming and outgoing edge lengths at a randomly chosen vertex. Second, we study the behavior of chemical distances on scale-free Gilbert graphs and show the existence of different regimes depending on the tail index of the radius distribution. Despite some similarities to long-range percolation and ultra-small scale-free geometric networks, scale-free Gilbert graphs are actually more closely related to fractal percolation and this connection gives rise to different scaling limits. We also propose a modification of the graph, where the total number of edges can be reduced substantially at the cost of introducing a logarithmic factor in the chemical distances.
\end{abstract}

\maketitle

\section{Introduction}
\label{intSec}
The spatial distribution of the population in a country is typically far from homogeneous, but rather exhibits fractal patterns. Specifically, this has been investigated for Great Britain and the United States~\cite{aby1,aby2} and for Finland~\cite{koski}. It is pointed out in~\cite{aby1,koski} that fractality has important implications for the design of wired telecommunication networks in the sense that the number and extent of various levels of hierarchy should be adapted to the fractal geometry. A trade-off is involved in determining the optimal number of levels. Using few levels has the advantage that most access points can be connected by a small number of hops. However, this comes at the cost of having to install a cable network of large total length. Indeed, a high-level node has to be connected to a large number of low-level nodes in order to guarantee connectivity. Vice versa, using a large number of hierarchies, one may be able to reduce the length of cables substantially, but this comes at the cost of increasing the number of hops it takes a low-level node to reach the topmost layer of hierarchy. A related cost analysis for hierarchical spatial networks is provided by Baccelli and Zuyev~\cite{baccelliZuyev}. However, the network structure investigated in that paper is obtained by iteratively considering Voronoi tessellations and does not exhibit fractal geometries (see however~\cite{aggTess} for some results in this direction).

In order to develop a fundamental understanding of the asymptotic behavior of cable lengths and chemical distances (i.e., minimal number of hops needed to connect two points) in large networks, it is important to abstract from the types of specific countries and deterministic fractals investigated in~\cite{aby1,koski} and move to random networks. The spatial nature of the problem calls for models based on random geometric graphs, and the fractal geometry suggests that one should look for scale-free random graphs exhibiting a power-law degree distribution. The combination of these two constraints restricts the list of appropriate choices substantially. We briefly review some of the most well-known models in literature and discuss their drawbacks with respect to modeling the kind of networks we have in mind. 

One option could be to use long-range percolation, see~\cite{lrp1,lrp2,copper} and the references therein. Here, one starts from the lattice $\Z^d$ and connects any pair of sites independently with a probability depending only on their distance. This leads to a network with a giant connected component and power-law degree distributions. However, this model does not offer an inherently defined hierarchy of nodes. \Mo, if we think of a high-level node as having the purpose of providing access to all low-level nodes in some region, then this would suggest that the occurrence of edges should be spatially positively correlated. However, if $x,y\in\Z^d$ are arbitrary sites in $\Z^d$ and $y^\p$ is close to $y$, then putting an edge between $x$ and $y$ does not influence at all the probability of seeing an edge between $x$ and $y^\p$. 

A second option could be to consider the ultra-small scale-free geometric networks on $\Z^d$ introduced in~\cite{ultraSmall}. Two sites $x,y\in\Z^d$ are connected by an edge in this graph if $|x-y|\le \min\{R_x,R_y\}$, where $\{R_z\}_{z\in\Z^d}$ denotes a family of iid heavy-tailed random variables. For any site $z\in\Z^d$ the value $R_z$ can be thought of as the radius of influence of $z$, \sot ultra-small scale-free geometric networks offer a natural possibility for defining the network hierarchy. Similar remarks apply to scale-free percolation~\cite{sfPerc}. Recently, also a Poisson-based continuum analogue of this model has been investigated, see~\cite{wuth1,wuth2}.
From a modeling point of view, this means that there is an excellent degree of connectivity between high-level nodes. However, conversely, for low-level nodes it may be very difficult to get connected to a nearby high-level node. This might pose a substantial obstruction to the build-up of a hierarchical network. Due to the drawbacks of the existing networks models, we propose two alternatives.


According to our discussion, it would be desirable to consider a variant of the ultra-small scale-free geometric network, where two sites $x,y\in\Z^d$ are connected by an edge if $|x-y|\le \max\{R_x,R_y\}$ (instead of $|x-y|\le \min\{R_x,R_y\}$). In other words, in order to connect $x$ and $y$ by an edge it is no longer necessary that both $x$ lies within the radius of influence of $y$ \emph{and} $y$ lies within the radius of influence of $x$. It suffices that $x$ lies within the radius of influence of $y$ \emph{or} $y$ lies within the radius of influence of $x$. 
We consider a spatial variant of this network model, called \emph{scale-free Gilbert graph}, where the vertices are given by a homogeneous Poisson point process on the $d$-dimensional torus with side length $n$, for some $d,n\ge2$. See Figure~\ref{sfbmFig} for an illustration of this graph in dimension $d=2$.

\begin{figure}[!htpb]
\centering
\begin{subfigure}{0.45\textwidth}
\centering
 {\includegraphics[width=6.5cm]{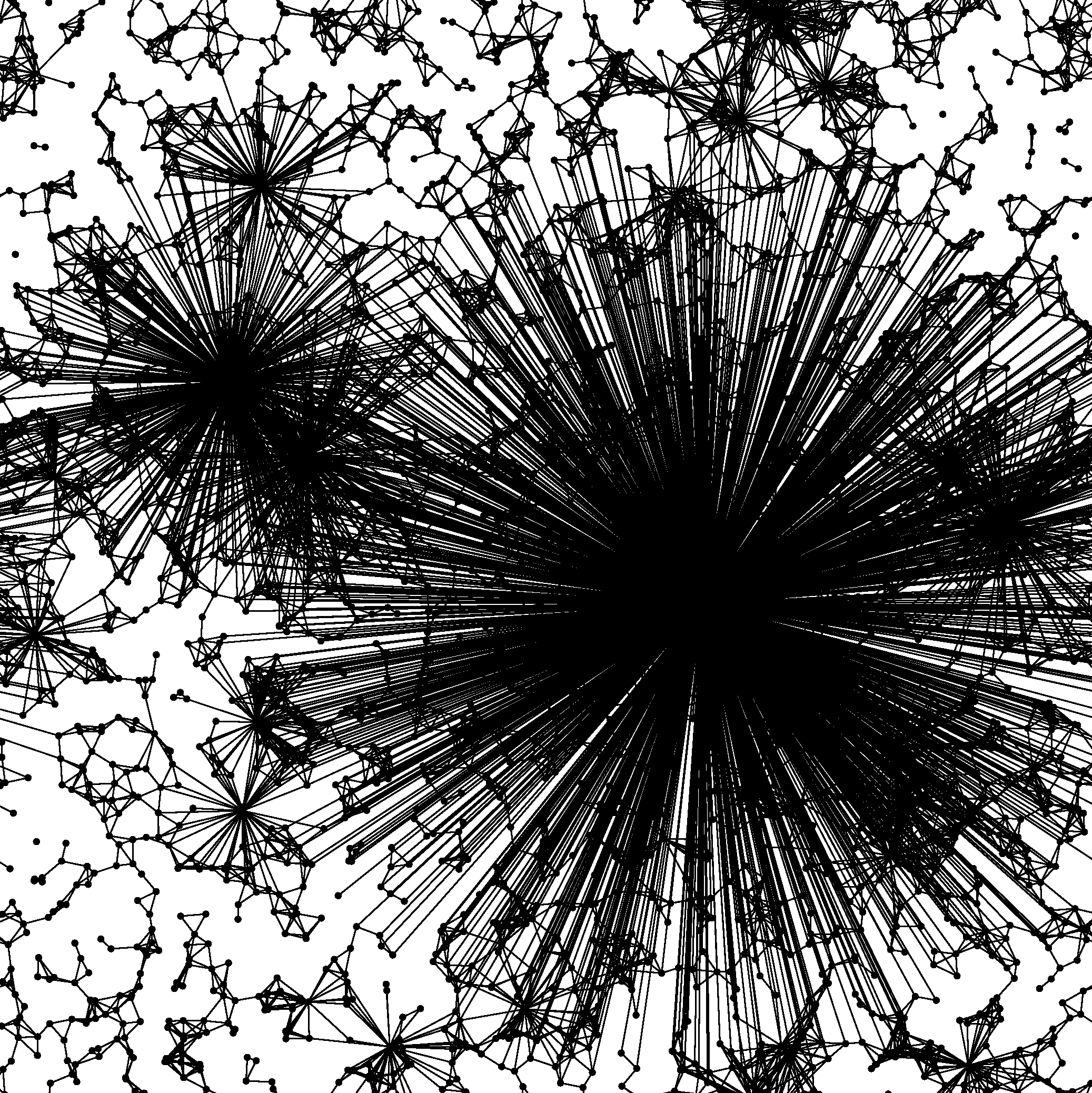}}
\caption{Scale-free Gilbert graph}\label{sfbmFig}
\end{subfigure}\qquad
        \begin{subfigure}{0.45\textwidth}
\centering
 {\includegraphics[width=6.5cm]{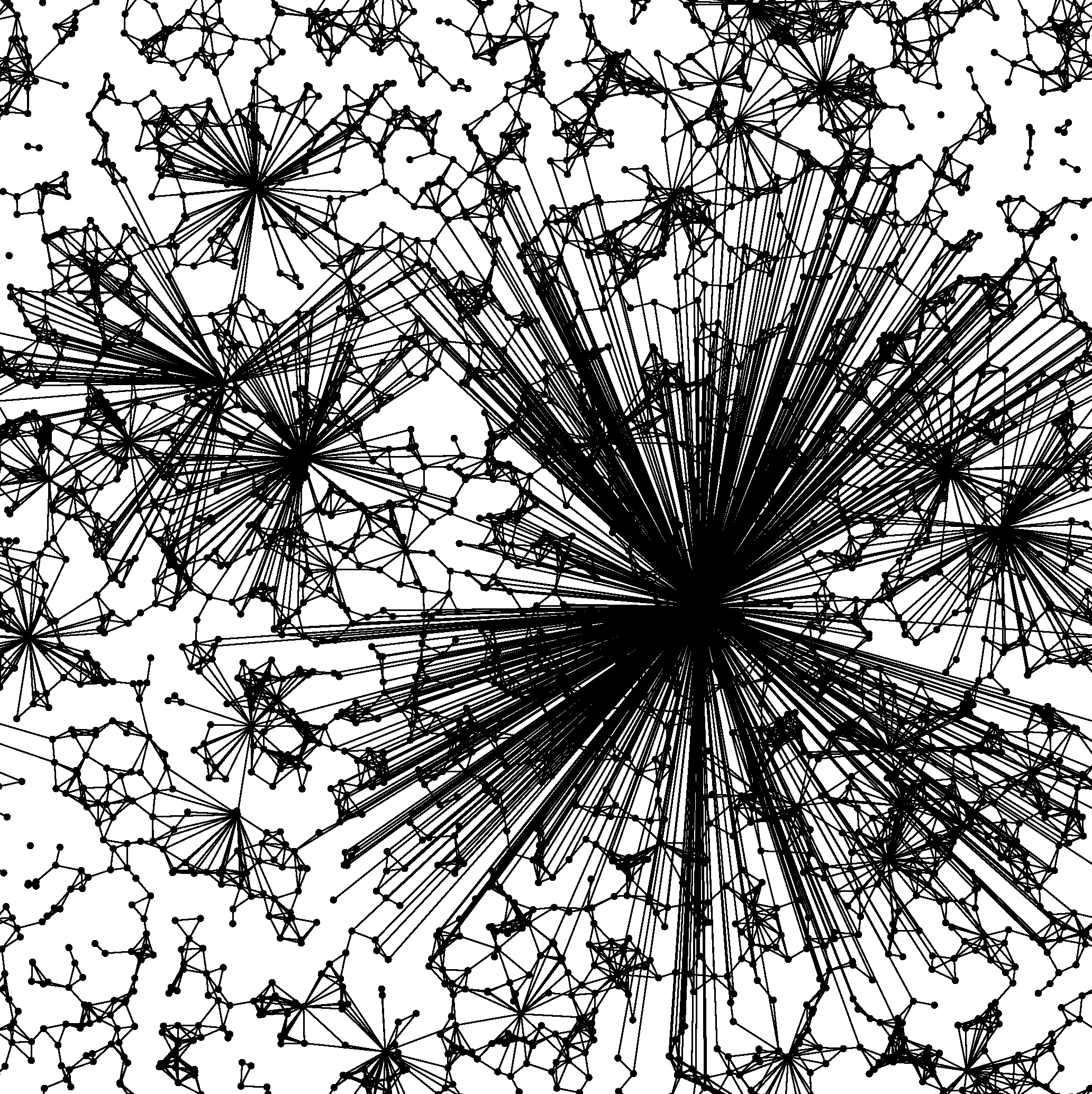}}
\caption{Thinned scale-free Gilbert graph}\label{onnFig}
\end{subfigure}
\caption{Planar scale-free and thinned scale-free Gilbert graphs}
\end{figure}

As $n\to\infty$, we investigate the asymptotic distribution of the power-weighted sum of all incoming and all outgoing edge lengths considered from a vertex that is picked uniformly at random, see Theorems~\ref{outEdgeThm} and~\ref{inEdgeThm}. In particular, our results imply that the asymptotic out-degree and in-degree distributions admit polynomial tails. We also investigate the growth of the expected power-weighted sum of all outgoing edge lengths as $n\to\infty$, see Theorem~\ref{asyGrowProp}. In Section~\ref{chemDistSec}, we show that different scaling regimes of chemical distances emerge depending on the tail index of the radius distribution.

Since the scale-free Gilbert graph exhibits a large degree of redundancy of connections, in Section~\ref{multiSec} we present a variant, called \emph{thinned scale-free Gilbert graph}, that substantially decreases these redundancies by introducing a multi-layer topology. Loosely speaking, an edge in the original graph from a low-level node $x$ to a high-level node $y$ is removed if $y$ can also be reached from $x$ via an intermediate node $z$. See Figure~\ref{onnFig} for an illustration of the thinned scale-free Gilbert graph.
A precise definition of this graph will be given in Section~\ref{defSec}, but we note already at this point that removing redundancies does not alter the family of connected components. \Fm, under suitable assumptions on the tail behavior of the radius distribution, we show in Theorem~\ref{redElThm} that the sum of power-weighted edge lengths at a randomly picked vertex decreases significantly. Of course, reducing edge lengths comes at the price of longer connection paths, but we will show in Theorem~\ref{intMedProp} that chemical distances grow at most by a logarithmic factor in the size of the torus.

The present paper is organized as follows. In Section~\ref{defSec}, we provide precise definitions of the scale-free and thinned scale-free Gilbert graphs and state our main results. Next, in Section~\ref{sfSec}, we investigate power-weighted sums of edge lengths at a typical point and prove Theorems~\ref{outEdgeThm},~\ref{inEdgeThm} and~\ref{asyGrowProp}. Section~\ref{chemDistSec} is devoted to the analysis of chemical distances for different regimes of the tail index. Finally, in Section~\ref{multiSec}, we investigate the changes in terms of the asymptotic behavior of power-weighted sums of edge lengths and of chemical distances as we move from the scale-free Gilbert graph to its thinning.

\section{Model definition and statement of main results}
\label{defSec}
In Sections~\ref{htggSec} and~\ref{thtggSec}, we provide precise definitions of the network models under consideration and present our main results. In the following, $d\ge2$ is always assumed to be an arbitrary fixed integer. \Mo, $B_r(\xi)=\{\eta\in\R^d:\,|\eta-\xi|\le r\}$ denotes the ball of radius $r>0$ centered at $\xi\in\R^d$.

\subsection{Scale-free Gilbert graph}
\label{htggSec}
Let $X^{(n)}$ be an independently $[0,\infty)$-marked homogeneous Poisson point process with intensity $1$ in the torus $\T_n$, where $\T_n$ is obtained from the cube $[-n/2,n/2]^d$ by the standard identification of its boundaries. If $n\ge1$, $r>0$ and $\xi\in\T_n$, then we write $\BT_r(\xi)=\{\eta\in\T_n:\d_{\T_n}(\xi,\eta)\le r\}$ for the closed ball in $\T_n$ with radius $r$ centered at $\xi$, where $\d_{\T_n}(\xi,\eta)$ denotes the toroidal distance between $\xi$ and $\eta$. 
The mark of a point from $X^{(n)}$ is interpreted as the radius of a ball centered at this point. Throughout the paper, we assume that the distribution of the typical mark $R$ is absolutely continuous and heavy-tailed. That is, \te $\beta,s\in(0,\infty)$ \st $\lim_{h\to\infty} h^{s}\P(R>h)=\beta$. For every $\varepsilon\in(0,1)$, we also fix $t_0=t_0(\varepsilon)>0$ \st $t^{s}\P(R>t)\in (\beta(1-\varepsilon),\beta(1+\varepsilon))$ \fa $t\ge t_0(\varepsilon)$.


For $n\ge1$ we investigate the directed random geometric graph $G_{}(X^{(n)})$ on the vertex set $X$, where an edge in $G(X^{(n)})$ is drawn from $x=(\xi,r)\in X^{(n)}$ to $y=(\eta,t)\in X^{(n)}$ if $\eta\in\BT_r(\xi)$. 

First, we investigate the asymptotic distributions of the power-weighted sum of all outgoing and all incoming edge lengths. To be more precise, fix $\alpha\ge0$, $n\ge1$ and let $R^*$ be a copy of $R$ that is independent of $X^{(n)}$. Then, 
$$D^{(\alpha)}_{\mathsf{out},n}=\sum_{(\xi,r)\in X^{(n)}}|\xi|^\alpha1_{\BT_{R^*}(o)}(\xi),$$
denotes the sum of $\alpha$th powers of the lengths of all outgoing edges at the node $(o,R^*)$, where we put $|\xi|=\d_{\T_n}(\xi,o)$. Considering the limit $n\to\infty$, i.e., letting the size of the torus tend to infinity, it is intuitive (and will be shown rigorously in Theorem~\ref{outEdgeThm} below) that the random variables $\big\{D^{(\alpha)}_{\mathsf{out},n}\big\}_{n\ge1}$ converge in distribution to the random variable 
$$D^{(\alpha)}_{\mathsf{out}}=\sum_{(\xi,r)\in X}|\xi|^\alpha1_{B_{R^*}(o)}(\xi),$$
where $X$ denotes an independently $[0,\infty)$-marked homogeneous Poisson point process in $\R^d$ with intensity $1$. In addition to showing the convergence of $D^{(\alpha)}_{\mathsf{out},n}$ to $D^{(\alpha)}_{\mathsf{out}}$, we also investigate the behavior of the tail probabilities $p^{(\alpha)}_{\mathsf{out},t}=\P\big(D^{(\alpha)}_{\mathsf{out}}>t\big)$ as $t\to\infty$. In the following, $\kappa_d$ denotes the volume of the unit ball in $\R^d$.
\begin{theorem}
\label{outEdgeThm}
For every $\alpha>0$ the random variables $\big\{D^{(\alpha)}_{\mathsf{out},n}\big\}_{n\ge1}$ converge to the random variable $D^{(\alpha)}_{\mathsf{out}}$ in distribution. Moreover, 
$$\lim_{t\to\infty}t^{s/(\alpha+d)}p^{(\alpha)}_{\out,t}=(d\kappa_d/(\alpha+d))^{s/(\alpha+d)}\beta.$$
\end{theorem}

\Ip, the degree distribution of out-degrees admits asymptotically polynomial tails of order $s/d$. In contrast, when considering in-degrees an entirely different asymptotic behavior emerges. Indeed, in Theorem~\ref{inEdgeThm} below we show that for $s\le d$ the degree distribution is asymptotically degenerate, whereas for $s>d$ it is Poissonian (with finite mean). In order to make this precise, let 
$$D^{(\alpha)}_{\inn,n}=\sum_{(\xi,r)\in X^{(n)}}|\xi|^\alpha1_{\BT_r(o)}(\xi),$$
denote the sum of $\alpha$th powers of the lengths of all incoming edges at the node $(o,R^*)$, where $\alpha\ge0$, $n\ge1$. As in the case of $D^{(\alpha)}_{\out,n}$, we will see that the random variables $D^{(\alpha)}_{\inn,n}$ converge in distribution to a random variable $D^{(\alpha)}_{\inn}$ given by
$$D^{(\alpha)}_{\inn}=\sum_{(\xi,r)\in X}|\xi|^\alpha1_{B_r(o)}(\xi),$$
where as before $X$ denotes an independently $[0,\infty)$-marked homogeneous Poisson point process in $\R^d$ with intensity $1$. The following result is devoted to the tail probabilities $p^{(\alpha)}_{\mathsf{in},t}=\P\big(D^{(\alpha)}_{\inn}>t\big)$
\begin{theorem}
\label{inEdgeThm}
For every $\alpha\ge0$ the random variables $\big\{D^{(\alpha)}_{\inn,n}\big\}_{n\ge1}$ converge to the random variable $D^{(\alpha)}_{\inn}$ in distribution. 
\Mo, $\P\big(D^{(\alpha)}_{\inn}=\infty\big)=1$ for $s\le d$, whereas if $s>d$, then $D^{(0)}_{\inn}$ is a Poissonian random variable with mean $\kappa_d\E R^d$. Finally, if $\alpha>0$ and $s>d$, then  
$$\lim_{t\to\infty} t^{(s-d)/\alpha}p^{(\alpha)}_{\inn,t}=d\kappa_d\beta (s-d)^{-1}.$$
\end{theorem}

Besides considering limit distributions, we also determine leading-order asymptotics for the expectations $\big\{\E D^{(\alpha)}_{\inn,n}\big\}_{n\ge1}$ as $n\to\infty$. Note that here it is not necessary to distinguish between ingoing and outgoing edges, since $\E D^{(\alpha)}_{\out,n}=\E D^{(\alpha)}_{\inn,n}$ for all $n\ge1$. 
\begin{theorem}
\label{asyGrowProp}
Let $\alpha>0$ be arbitrary. Then one can distinguish between three limiting regimes for the random variable $\E D^{(\alpha)}_{\inn,n}$ depending on the sign of $\alpha+d-s$.
\begin{enumerate}
\item If $s>\alpha+d$, then $\lim_{n\to\infty}\E D^{(\alpha)}_{\inn,n}=d\kappa_d(\alpha+d)^{-1}\E R^{\alpha+d}$.
\item If $s=\alpha+d$, then $\lim_{n\to\infty}(\log n)^{-1}\E D^{(\alpha)}_{\inn,n}=d\kappa_d\beta$.
\item If $s<\alpha+d$, then $\lim_{n\to\infty}n^{s-\alpha-d}\E D^{(\alpha)}_{\inn,n}=\beta  \int_{[-1/2,1/2]^d}|\eta|^{\alpha-s}\d \eta$.
\end{enumerate}
\end{theorem}

Next, we investigate the asymptotic behavior of chemical distances in scale-free Gilbert graphs, where we allow a directed edge in $G(X^{(n)})$ to be traversed in both directions. For the analysis of $G(X^{(n)})$, it will be convenient to distinguish between the regimes $s<d$, $s=d$ and $s>d$. We say that a family of events occurs \emph{with high probability (whp)} if the probabilities of the events tend to $1$ as $n\to\infty$. In the case where $s<d$, whp there exists some point of $X^{(n)}$ that is connected to all other points by an edge in $G(X^{(n)})$. \Ip, the diameter of $G(X^{(n)})$ is at most $2$ whp. The cases $s=d$ and $s>d$ are more subtle. If $s>d$, then as $n\to\infty$ the probability that $G(X^{(n)})$ contains isolated vertices is bounded away from $0$, see Proposition~\ref{nonConnLem}. Still, if $\beta$ is \suf large, then it follows from continuum percolation that whp there exists a giant component containing a positive proportion of the vertices. But even inside this giant connected component the effects of the heavy-tailed nature of $R$ are barely noticeable. This is to be understood in the sense that chemical distances (i.e., minimal number of hops) between nodes at distance $n$ grow almost linearly in $n$. To be more precise, putting $\mathsf{e}_1=(1,0,\ldots,0)\in\T_n$ and denoting by $q(x)$ the closest point of $X$ seen from a given point $x\in\T_n$, we have the following result.
\begin{theorem}
\label{subCrit}
Assume that $s>d$ and let $\alpha>0$ be arbitrary. Then, the chemical distance between $q(-n\mathsf{e}_1/4)$ and $q(n\mathsf{e}_1/4)$ is at least $n/(\log n)^{\alpha}$ whp.
\end{theorem}
We conjecture that the sublogarithmic correction factor is only an artifact of our proof.
\begin{conjecture}
Assume that $s>d$. Then, there exists a constant $c=c(\beta,d)>0$ \st the chemical distance between $q(-n\mathsf{e}_1/4)$ and $q(n\mathsf{e}_1/4)$ is at least $cn$ whp.
\end{conjecture}

Finally, we assume that $s=d$. On the one hand, in contrast to the case $s>d$ the heavy-tailedness of the radius distribution has a substantial effect on chemical distances. On the other hand, the connectivity structure still does not degenerate as in the case $s<d$. Therefore, from the point of view of modeling scale-free network structures, the regime $s=d$ might be considered to be the most interesting one. In Theorem~\ref{sd1Thm} below, we show that if the parameter $\beta$ is sufficiently large, then not only the graph $G(X^{(n)})$ is connected whp, but, moreover, its diameter remains bounded in the sense that whp it is dominated by the total progeny of a subcritical Galton-Watson process. This asymptotic behavior differs decisively from the one of either long-range percolation~\cite{lrp1,lrp2} or ultra-small scale-free geometric networks~\cite{ultraSmall}. As we will see in the proof of Theorem~\ref{sd1Thm}, this scaling regime is a consequence of the close relationship between scale-free Gilbert graphs and fractal percolation processes.
\begin{theorem}
\label{sd1Thm}
Assume that $\beta>d^{d/2}2^{2d+1}(d+1)\log 2$. Then, the graph $G(X^{(n)})$ is connected whp, and, moreover, the diameter $\mathsf{diam}\,G(X^{(n)})$ of $G(X^{(n)})$ is stochastically dominated by an affine function in the total progeny of a subcritical Galton-Watson process whp. To be more precise, \tes a coupling between $X^{(n)}$ and the total progeny $T$ of a subcritical Galton-Watson process \st $\P(\mathsf{diam}\,G(X^{(n)})>2+2^dT)\in O(n^{-1})$.
\end{theorem}

\subsection{Thinned scale-free Gilbert graph}
\label{thtggSec}
As explained in Section~\ref{intSec}, the scale-free Gilbert graph $G(X^{(n)})$ contains a large number of redundant edges, which can be removed without affecting its connected components. To be more precise, if $x_1=(\xi_1,r_1),x_2=(\xi_2,r_2),x_3=(\xi_3,r_3)\in X^{(n)}$ are \st $r_1>r_2>r_3$, $\xi_2\in \BT_{r_1}(\xi_1)$ and $\xi_3\in \BT_{r_1}(\xi_1)\cap \BT_{r_2}(\xi_2)$, then in $G(X^{(n)})$ the point $x_3$ is connected both to $x_1$ and $x_2$. However, the edge from $x_3$ to $x_1$ is redundant since one can also reach $x_1$ from $x_3$ by first moving from $x_3$ to $x_2$ and then from $x_2$ to $x_1$. See Figure~\ref{thinCon} for an illustration of this configuration.
In order to reduce the total network length, we therefore introduce a variant $G^\p$ of the original graph $G$, where such redundancies are removed. 

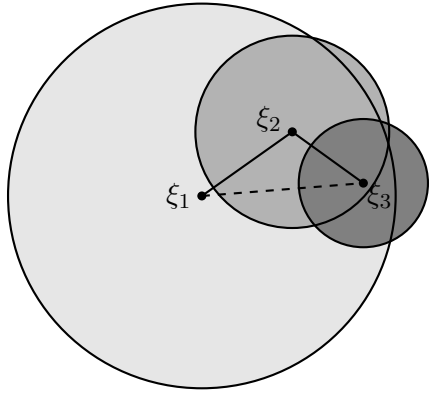
\begin{figure}[!htpb]
\centering
\begin{tikzpicture}[scale=0.85]
\fill[black!10!white] (0,0) circle (3cm);
\fill[black!30!white] (1.4,1.0) circle (1.5cm);
\fill[black!50!white] (2.5,0.2) circle (1cm);

\draw[thick] (0,0) circle (3cm);
\draw[thick] (1.4,1.0) circle (1.5cm);
\draw[thick] (2.5,0.2) circle (1cm);

\fill (0,0) circle (2pt);
\fill (1.4,1.0) circle (2pt);
\fill (2.5,0.2) circle (2pt);

\draw[thick] (0,0) --(1.4,1.0)--(2.5,0.2);
\draw[thick,dashed] (0,0) -- (2.5,0.2);

\coordinate[label=180:$\xi_1$] (u) at (0.0,0.0);
\coordinate[label=180:$\xi_2$] (u) at (1.4,1.2);
\coordinate[label=0:$\xi_3$] (u) at (2.4,0.0);

\end{tikzpicture}
\caption{The dashed edge connecting $x_1$ and $x_3$ is redundant}
\label{thinCon}
\end{figure}

To be more precise, for any finite subset $\varphi$ of $\T_n\times[0,\infty)$ define the \emph{thinned scale-free Gilbert graph} $G^\p(\varphi)$ as the graph on $\varphi$, where an edge is drawn from $(\xi,r)\in\varphi$ to $(\eta,t)\in\varphi$ if $(\eta,t)\in \BT_r(\xi)\times[0,r)$ and there does not exist $(\zeta,w)\in\varphi\cap (\BT_r(\xi)\times(t,r))$ \st $\eta\in\BT_w(\zeta)$.

By definition, $G^\p(X^{(n)})$ is a subgraph of $G(X^{(n)})$ and we will see in Proposition~\ref{conProp} that the thinning does not affect the connected components of $G(X^{(n)})$. From the point of view of telecommunication networks, we can reach the same set of subscribers using a smaller cable length. Next, we investigate the question whether the reduction of cable length is substantial. For concreteness, we assume that $s=d$. We show that the leading order of the expected power-weighted sum of lengths of outgoing edges in $G^\p(\{(o,R^*)\}\cup X^{(n)})$ is strictly smaller when compared to the graph $G(X^{(n)})$, see Theorem~\ref{asyGrowProp}. 
To be more precise, defining
$$D^{\p,(\alpha)}_{\out,n}=\sum_{(\xi,r)\in X^{(n)}}|\xi|^\alpha1_{((o,R^*),(\xi,r))\text{ is an edge in }G^\p(\{(o,R^*)\}\cup X^{(n)})},$$
where $\alpha\ge 0$  and $n\ge1$ we have the following result. 
\begin{theorem}
\label{redElThm}
Assume that $s=d$. Then, the expected out-degree $\E D^{\p,(0)}_{\out,n}$ is asymptotically bounded from above by a constant, i.e., $\E D^{\p,(0)}_{\out,n}\in O(1)$. \Mo, if $\alpha>0$, then $\E D^{\p,(\alpha)}_{\out,n}\in O(n^{\alpha-\delta})$ for some $\delta>0$.
\end{theorem}
In Section~\ref{multiSec}, we show that removing redundancies does not destroy the property of connectivity. However, this thinning operation \emph{does} influence the quality of connectivity, in the sense that chemical distances will increase. Indeed, instead of moving from $x\in X^{(n)}$ to $y\in X^{(n)}$ directly along an edge in the graph $G(X^{(n)})$, introducing a multilayer topology via $G^\p(X^{(n)})$ might force us to move through a potentially large number of layers, before we can get from $x$ to $y$. Still, chemical distances increase at most by a logarithmic factor in the size of the sampling window.
\begin{theorem}
\label{intMedProp}
There exists $c_1>0$ \st whp any $x,y\in X^{(n)}$ that are adjacent in $G(X^{(n)})$ can be connected by a path in $G^\p(X^{(n)})$ consisting of at most $c_1\log n$ hops.
\end{theorem}

\section{Proofs of Theorems~\ref{outEdgeThm},~\ref{inEdgeThm} and ~\ref{asyGrowProp}}
\label{sfSec}
In the present section, we investigate the asymptotic distributions of the sum of all outgoing and of the sum of all incoming power-weighted edge lengths at a typical vertex. First, in Theorem~\ref{outEdgeThm}, we consider the case of outgoing edges. The proof is essentially based on the observation that conditioned on $R^*$ the random variable $D^{(\alpha)}_{\out}$ concentrates around its conditional mean.

\begin{proof}[Proof of Theorem~\ref{outEdgeThm}]
Using the canonical coupling between the Poisson point process $X^{(n)}$ on the torus $\T_n$ and the Poisson point process $X$ on $\R^d$ we deduce that
$$\P\big(D^{(\alpha)}_{\mathsf{out},n}\ne D^{(\alpha)}_{\out}\big)\le \P(R^*>n/2),$$
and the latter probability tends to $0$ as $n\to\infty$. In order to determine the tail behavior of the random variable $D^{(\alpha)}_{\out}$, for any $r>0$ we put 
$$D_r=\sum_{(\xi,t)\in X}|\xi|^\alpha1_{B_r(o)}(\xi),$$
noting that $D_{R^*}=D^{(\alpha)}_{\out}$.
\Ip, 
\begin{align}
\label{inEdgeEq1}
t^{s/(\alpha+d)}p^{(\alpha)}_{\out,t}=\int_0^\infty t^{s/(\alpha+d)}\P(D_r>t)\P_R(\d r),
\end{align}
where $\P_R$ denotes the distribution of the random variable $R$. Let $\varepsilon>0$ be arbitrary. In order to analyze~\eqref{inEdgeEq1}, we first consider the case where $t\ge (v(\varepsilon)^{-1}r)^{\alpha+d}$, i.e., where $r\le v(\varepsilon) t^{1/(\alpha+d)}$, writing $v(\varepsilon)=(1+\varepsilon)^{-1}(d\kappa_d/(\alpha+d))^{-1/(\alpha+d)}$. Let $N_r$ be a Poissonian random variable with mean $\kappa_dr^d$, where $r>0$. Additionally, $\{U_i\}_{i\ge1}$ be an iid sequence of random vectors that are independent of $N_r$ and uniformly distributed in $B_1(o)$, \sot $D_r=r^\alpha\sum_{i=1}^{N_r} |U_i|^{\alpha}$. First, 
$$\P(D_r>t)\le \P(N_r\ge n_{r,t})+\P\Big(\sum_{i=1}^{n_{r,t}}r^{\alpha}|U_i|^{\alpha}\ge t\Big),$$
where $n_{r,t}=\lfloor tr^{-\alpha}\kappa_dv(\varepsilon)^{\alpha+d}(1+\varepsilon)^{(\alpha+d)/2}\rfloor$. Note that for sufficiently large values of $tr^{-\alpha}$ we have
$$n_{r,t}/(\kappa_dr^d)\ge t(v(\varepsilon)r^{-1})^{\alpha+d}(1+\varepsilon)^{(\alpha+d)/4}\ge (1+\varepsilon)^{(\alpha+d)/4},$$
\sot the Poisson concentration property~\cite[Lemma 1.2]{penrose} implies that $\sup_{r\le v(\varepsilon) t^{1/(\alpha+d)}} \P(N_r\ge n_{r,t})$
decays at least exponentially fast in $t^{d/(\alpha+d)}$ as $t\to\infty$. Similarly, taking into account $\E|U_1|^\alpha=d/(\alpha+d)$, we obtain that 
$$tr^{-\alpha}/(n_{r,t}\E |U_1|^\alpha)\ge(1+\varepsilon)^{(\alpha+d)/2},$$
\sot using the classical theory of large deviations shows that also the expression
$$\sup_{r\le v(\varepsilon) t^{1/(\alpha+d)}}\P\Big(\sum_{i=1}^{n_{r,t}}r^{\alpha}|U_i|^{\alpha}\ge t\Big)$$
decays at least exponentially fast in $t^{d/(\alpha+d)}$ as $t\to\infty$. \Ip, 
\begin{align}
\label{outEdgeEq1}
\lim_{t\to\infty} t^{s/(\alpha+d)}p^{(\alpha)}_{\out,t}=\lim_{t\to\infty}\int_{v(\varepsilon)t^{1/(\alpha+d)}}^\infty t^{s/(\alpha+d)}\P(D_r>t)\P_{R}(\d r),
\end{align}
provided that the latter limit exists and is finite. In order to compute the right-hand side in~\eqref{outEdgeEq1}, we split the integral into three parts, which are analyzed separately. To be more precise, put
\begin{align*}
I_1&=\int_{v(\varepsilon)t^{1/(\alpha+d)}}^{v(-\varepsilon)t^{1/(\alpha+d)}}t^{s/(\alpha+d)}\P(D_r>t)\P_R(\d r),\\
I_{2,1}&=-\int_{v(-\varepsilon)t^{1/(\alpha+d)}}^\infty t^{s/(\alpha+d)}\P(D_r\le t)\P_R(\d r),
\end{align*}
and
$$I_{2,2}=\int_{v(-\varepsilon)t^{1/(\alpha+d)}}^\infty t^{s/(\alpha+d)}\P_R(\d r)=t^{s/(\alpha+d)}\P(R>v(-\varepsilon)t^{1/(\alpha+d)}).$$
First, note that $I_{2,2}$ tends to $v(-\varepsilon)^{-s}\beta$ as $t\to\infty$ and the latter expression tends to $(d\kappa_d/(\alpha+d))^{s/(\alpha+d)}$ as $\varepsilon\to0$. Hence, it suffices to show that the integrals $I_1$ and $I_{2,1}$ tend to $0$ as we let first $t\to\infty$ and then $\varepsilon\to0$. Indeed, 
$$I_1\le t^{s/(\alpha+d)}\big(\P(R>{v(\varepsilon)t^{1/(\alpha+d)}})-\P(R>{v(-\varepsilon)t^{1/(\alpha+d)}})\big),$$
and the right-hand side tends to $\beta(v(\varepsilon)^{-s}-v(-\varepsilon)^{-s})$ as $t\to\infty$, which vanishes as $\varepsilon\to0$. Finally, we observe that 
$$-I_{2,1}\le t^{s/(\alpha+d)}\P(D_{v(-\varepsilon)t^{1/(\alpha+d)}}\le t)\P(R>v(-\varepsilon)t^{1/(\alpha+d)}).$$
We conclude the proof of the theorem by noting that the expression $t^{s/(\alpha+d)}\P(R>v(-\varepsilon)t^{1/(\alpha+d)})$ remains bounded as $t\to\infty$, whereas $\P(D_{v(-\varepsilon)t^{1/(\alpha+d)}}\le t)$ tends to $0$. Indeed, $D_{v(-\varepsilon)t^{1/(\alpha+d)}}$ has expectation $(1-\varepsilon)^{-(\alpha+d)}t$ and variance
$$v(-\varepsilon)^{2\alpha+d}t^{(2\alpha+d)/(\alpha+d)} v_{2\alpha}^{\alpha+d},$$
 \sot the latter claim follows from the Chebyshev inequality.
\end{proof}
The proof of Theorem~\ref{inEdgeThm} relies on a rather delicate comparison of $D^{(\alpha)}_{\inn}$ and the random variable $\max_{(\xi,r)\in X}|\xi|^\alpha1_{B_r(o)}(\xi)$.

\begin{proof}[Proof of Theorem~\ref{inEdgeThm}]
First, observe that the in-degree of the origin $D^{(0)}_{\inn,n}$ is a Poissonian random variable with mean 
$$\int_{\T_n}\P(R>|\xi|)\d\xi,$$
which is at least ${d\kappa_d}\int_0^{n/2} r^{d-1}\P(R>r)\d r.$
If $s\le d$, then this lower bound tends to $\infty$ as $n\to\infty$. Using $D^{(\alpha)}_{\inn,n}\le D^{(\alpha)}_{\inn} $, this shows that $\P\big(D^{(\alpha)}_{\inn}=\infty\big)=1$ and that $(D^{(\alpha)}_{\inn,n})_{n\ge1}$ converge in distribution to $D^{(\alpha)}_{\inn}$. In the following, we may therefore assume that $s>d$. To show the assertion on the convergence in distribution, we proceed as in the proof of Theorem~\ref{inEdgeThm}. Indeed,
using the canonical coupling between the $[0,\infty)$-marked Poisson point process $X^{(n)}$ on the torus $\T_n$ and the $[0,\infty)$-marked Poisson point process $X$ on $\R^d$, we deduce that
$$\P(D^{(\alpha)}_{\inn,n}\ne D^{(\alpha)}_{\inn})\le \P(X\cap S\setminus ( B_{n/2}(o)\times[0,\infty))\ne\es),$$
where 
$$S=\{(\xi,r)\in\R^d\times[0,\infty):\, r>|\xi|\}.$$
Note that $\#(X\cap S\setminus (B_{n/2}(o)\times[0,\infty)))$ is Poissonian with mean
\begin{align*}
\int_{\R^d\setminus B_{n/2}(o)}\P(R>|\xi|) \d \xi&={d\kappa_d} \int_{n/2}^\infty r^{d-1}\P(R>r) \d r.
\end{align*}
The assumption $s>d$ implies that the latter integral tends to $0$ as $n\to\infty$. We also observe that by definition, $D^{(0)}_{\inn}$ is a Poissonian random variable with mean $\E\#(X\cap S)$. \Mo, this mean can be expressed as
\begin{align*}
\int_{\R^d}\P(R>|\xi|) \d \xi={d\kappa_d} \int_{0}^\infty r^{d-1}\P(R>r)\d r=\kappa_d \E R^d.
\end{align*}
It remains to determine the tail behavior of $D^{(\alpha)}_{\inn}$, for $\alpha>0$. The derivation of the lower bound is based on the elementary relation
\begin{align}
\P(X\cap S\setminus( B_{t^{1/\alpha}}(o)\times[0,\infty))\ne\es)=\P\(\max_{(\xi,r)\in X\cap S}|\xi|^\alpha\ge t\)\le\P\big(D^{(\alpha)}_{\inn}\ge t\big).
\end{align}
Let $\varepsilon>0$ be arbitrary. Note that the random variable $\#(X\cap S\setminus( B_{t^{1/\alpha}}(o)\times[0,\infty)))$ is Poissonian with mean
\begin{align*}
\int_{\R^d\setminus B_{t^{1/\alpha}}(o)}\P(R\ge|\xi|)\d\xi={d\kappa_d}\int_{t^{1/\alpha}}^\infty r^{d-1}\P(R\ge r)\d r,
\end{align*}
and observe that \fa $t>t_0(\varepsilon)$ the latter expression is bounded from below by
\begin{align*}
(1-\varepsilon){d\kappa_d}\beta\int_{t^{1/\alpha}}^\infty r^{d-s-1}\d r&=(1-\varepsilon)(s-d)^{-1}{d\kappa_d}\beta  {t^{(d-s)/\alpha}}.
\end{align*}
Therefore, 
\begin{align*}
\P\big(D^{(\alpha)}_{\inn}\ge t\big)&\ge1-\exp(-(1-\varepsilon)(s-d)^{-1}{d\kappa_d}\beta{t^{(d-s)/\alpha}})\\
&\ge(1-2\varepsilon)(s-d)^{-1}d\kappa_d\beta{t^{(d-s)/\alpha}},
\end{align*}
provided that $t>0$ is \suf large.
For the upper bound, we need more refined arguments, since in principle $D^{(\alpha)}_{\inn}$ could be larger than $t$, even if $\max_{(\xi,r)\in X\cap S}|\xi|^\alpha\le t$. To achieve the desired upper bound we show, loosely speaking, that the behavior of the sum $D^{(\alpha)}_{\inn}$ is already determined by its largest summand. Let $\varepsilon\in(0,1)$ be arbitrary. Then our upper bound is based on the inequality
\begin{align*}
\P(D^{(\alpha)}_{\inn}>t)&\le \P\(\max_{(\xi,r)\in X\cap S}|\xi|\ge(t(1-\varepsilon))^{1/\alpha}\)+\P(\#(X\cap S)\ge \varepsilon t^{1/4})\\
&+\P(\#(X\cap S\setminus (B_{t^{3/(4\alpha)}}(o)\times[0,\infty)))\ge 2),
\end{align*}
and in the rest of the proof the three summands on the right-hand side are considered separately. Using similar bounds as in the derivation of the lower bounds, we obtain that 
$$\lim_{t\to\infty}t^{(s-d)/\alpha}\P\(\max_{(\xi,r)\in X\cap S}|\xi|\ge(t(1-\varepsilon))^{1/\alpha}\)=d\kappa_d \beta(s-d)^{-1}((1-\varepsilon))^{(d-s)/\alpha}.$$
\Fm, $\P(\#(X\cap S)\ge \varepsilon t^{1/4})$ tends to $0$ exponentially fast in $t^{1/4}$ since $\#(X\cap S)$ is a Poissonian random variable with finite mean.

Finally, $\#(X\cap S\setminus (B_{t^{3/(4\alpha)}}(o)\times[0,\infty)))$ is a Poissonian random variable whose mean is at most $2d\kappa_d \beta(s-d)^{-1}t^{3(d-s)/(4\alpha)}$. Hence, the Poisson concentration property yields
\begin{align*}
\P(\#(X\cap S\setminus (B_{t^{3/(4\alpha)}}(o)\times[0,\infty)))\ge 2)\in O(t^{-3(s-d)/(2\alpha)}).
\end{align*}
Since $\varepsilon>0$ was arbitrary, this completes the proof of Theorem~\ref{inEdgeThm}.
\end{proof}
Theorems~\ref{outEdgeThm} and~\ref{inEdgeThm} show that on the distributional level, there is a substantial difference between power-weighted sums of outgoing and incoming edge lengths. However, when moving to the level of expectations, these differences disappear. In fact, the equality of expectations is an immediate consequence of Slivnyak's theorem, and is true for a much more general class of Poisson-based random geometric graphs. Still, for the convenience of the reader, we present some details.
\begin{proposition}
Let $\alpha>0$ be arbitrary. Then $\E D^{(\alpha)}_{\out,n}=\E D^{(\alpha)}_{\inn,n}$.
\end{proposition}
\begin{proof}
For $(\xi,r),(\eta,t)\in X^{(n)}$ we write $((\xi,r),(\eta,t))\in G(X^{(n)})$ if there is a directed edge from $(\xi,r)$ to $(\eta,t)$ in the graph $G(X^{(n)})$. 
Then an application of Slivnyak's theorem yields
\begin{align*}
n^d\E D^{(\alpha)}_{\out,n}&= \E\sum_{(\xi,r)\in X^{(n)}}\sum_{(\eta,t)\in X^{(n)}}\d_{\T_n}(\xi,\eta)^{\alpha}1_{((\xi,r),(\eta,t))\in G(X^{(n)})}\\
&= \E\sum_{(\eta,t)\in X^{(n)}}\sum_{(\xi,r)\in X^{(n)}}\d_{\T_n}(\xi,\eta)^{\alpha}1_{((\xi,r),(\eta,t))\in G(X^{(n)})}\\
&=n^d\E D^{(\alpha)}_{\inn,n}.
\end{align*}
\end{proof}
Now, we can proceed with the proof of Theorem~\ref{asyGrowProp}.
\begin{proof}[Proof of Theorem~\ref{asyGrowProp}]
We begin with part (i). Noting that for every $n\ge1$ the random variable $D^{(\alpha)}_{\inn,n}$ is stochastically dominated by the random variable $D^{(\alpha)}_{\inn}$ and that the random variables $\{D^{(\alpha)}_{\inn,n}\}$ converge to $D^{(\alpha)}_{\inn}$ in distribution, it suffices to show that $\E D^{(\alpha)}_{\inn}=d\kappa_d(\alpha+d)^{-1}\E R^{\alpha+d}$. Indeed, Campbell's formula implies that 
\begin{align*}
\E D^{(\alpha)}_{\inn}&=\int_{\R^d}|\xi|^\alpha \P(R>|\xi|)\d \xi= d\kappa_d \int_0^\infty r^{\alpha+d-1}\P(R>r)\d r=d\kappa_d(\alpha+d)^{-1}\E R^{\alpha+d}.
\end{align*}
For parts (ii) and (iii) we can also proceed by applying Campbell's formula. Indeed, let $\varepsilon>0$ be arbitrary and choose $t_0=t_0(\varepsilon)$ as in Section~\ref{htggSec}. For part (ii), we first obtain that the expressions
\begin{align*}
(\log n)^{-1}\int_{ \BT_{t_0}(o)}|\xi|^{\alpha}\P(R>|\xi|)\d \xi
\end{align*}
and
\begin{align*}
(\log n)^{-1}\int_{\T_n\setminus \BT_{n/2}(o)}|\xi|^{\alpha}\P(R>|\xi|)\d \xi&\le 2\beta(\log n)^{-1} \int_{\T_1\setminus B^{\T_1}_{1/2}(o)}|\eta|^{\alpha-s}\d \eta
\end{align*}
tend to $0$ as $n\to\infty$. \Fm,
\begin{align*}
(\log n)^{-1}\int_{\BT_{n/2}(o)\setminus \BT_{t_0}(o)}|\xi|^{\alpha}\P(R>|\xi|)\d \xi&\le (1+\varepsilon)(\log n)^{-1}d\kappa_d\beta \int_{t_0}^{n/2}r^{-1}\d r\\
&=(1+\varepsilon)(\log n)^{-1}d\kappa_d\beta (\log n-\log 2 -\log t_0).
\end{align*}
Since an analogous argument gives the corresponding lower bound, this completes the proof of part (ii). It remains to deal with part (iii). Proceeding similarly to part (ii), we first note that 
\begin{align*}
n^{s-\alpha-d}\int_{\BT_{t_0}(o)}|\xi|^{\alpha}\P(R>|\xi|)\d \xi
\end{align*}
tends to $0$ as $n\to\infty$. \Fm,
\begin{align*}
n^{s-\alpha-d}\int_{\T_n\setminus \BT_{t_0}(o)}|\xi|^{\alpha}\P(R>|\xi|)\d \xi&\le (1+\varepsilon)\beta\int_{\T_1\setminus B^{\T_1}_{t_0/n}(o)}|\eta|^{\alpha-s}\d \eta,
\end{align*}
and the right-hand side tends to $(1+\varepsilon)\beta\int_{[-1/2,1/2]^d}|\eta|^{\alpha-s}\d\eta$ as $n\to\infty$. Again, since the lower bound can be obtained using similar arguments, this completes the proof of part (iii).
\end{proof}

\section{Chemical distances}
\label{chemDistSec}
In the present section we investigate the behavior of chemical distances (i.e., shortest-path lengths) on scale-free Gilbert graphs. We have already mentioned in the introduction that the regime, where $s<d$ might be of limited interest, as the diameter is at most $2$ whp. 
\subsection{Regime $s>d$}
\label{subCritReg}
In the present subsection, we consider the regime $s>d$ whose connectivity properties turn out to be rather similar to those of the Boolean model with light-tailed radii. For instance, we note that isolated vertices may occur with positive probability.
\begin{proposition}
\label{nonConnLem}
Assume that $s>d$. Then 
$$\lim_{n\to\infty}\P(o\text{ is isolated in }G(X^{(n)}\cup\{(o,R^*)\}))>0.$$
\end{proposition}
\begin{proof}
First, we note that similar arguments as in Theorems~\ref{outEdgeThm} and~\ref{inEdgeThm} can be used to show that the probabilities $\P(o\text{ is isolated in }G(X^{(n)}\cup\{(o,R^*)\}))$ converge as $n\to\infty$. \Mo, the events $\big\{D^{(0)}_{\inn,n}=0\big\}$ and $\big\{D^{(0)}_{\out,n}=0\big\}$ are both decreasing events, \sot the FKG inequality (see, e.g.~\cite[Theorem 1.4]{fock}) implies that they are positively correlated. Hence,
\begin{align*}
\P(o\text{ is isolated in }G(X^{(n)}\cup\{(o,R^*)\}))&\ge\P\big(D^{(0)}_{\inn,n}=0\big)\P\big(D^{(0)}_{\out,n}=0\big),
\end{align*}
\sot using Theorems~\ref{outEdgeThm} and~\ref{inEdgeThm} completes the proof.
\end{proof}

Next, we prove Theorem~\ref{subCrit}, which shows that allowing random radii with tail index $s>d$ does not reduce substantially chemical distances in comparison to the case of constant radii. To be more precise, the reduction amounts at most to a sublogarithmic factor. The key idea for the proof of Theorem~\ref{subCrit} is to analyze the connections in the graph $G(X^{(n)})$ at different length scales. Although at every scale the presence of long edges can be used to reduce chemical distances, still such shortcuts are sufficiently rare to yield only a sublogarithmic reduction factor in comparison to the linear growth rate $n$.
Before going into the (technical) details, we provide a rough sketch of the proof. Put $p=(s+d)/(2s)$, and subdivide the torus $\T_n$ into $k_1=n^{(1-p)d}$ subcubes $Q_1,\ldots,Q_{k_1}$ of side length $n^p$. The probability that \tes $(\xi,r)\in X^{(n)}$ with $r\ge n^p$ is of order at most $n^dn^{-sp}=n^{-(s-d)/2}$. On the other hand, if such a point does not exist, then the endpoints of any edge in $G(X^{(n)})$ are contained in adjacent subcubes. Hence, the number of subcubes that we need to visit if we move from $q(-n\mathsf{e}_1/4)$ to $q(n\mathsf{e}_1/4)$ is at least of order $n^{1-p}$. 

Now, we continue by subdividing each of the subcubes $Q_1,\ldots,Q_{k_1}$ into subsubcubes of side length $n^{p^2}$. For each of the subcubes $Q_i$ the probability that \tes $(\xi,r)\in X^{(n)}\cap Q_i$ with $r\ge n^{p^2}$ is of order at most $n^{pd}n^{-sp^2}=n^{-(s-d)p/2}$. The part of the path connecting $q(-n\mathsf{e}_1/4)$ and $q(n\mathsf{e}_1/4)$ in any such subcube is a nearest-neighbor path on the level of subsubcubes, \sot typically $n^{p-p^2}$ steps are needed to cross that subcube. If this was true \fa $i\in\{1,\ldots,k_1\}$, then we would obtain a lower bound for the chemical distance between $q(-n\mathsf{e}_1/4)$ and $q(n\mathsf{e}_1/4)$ that is of order $n^{1-p^2}$ and continuing in this fashion, we would obtain in fact a lower bound that is linear in $n$. However, at each level we have to deal with a small loss, which leads to the sublogarithmic correction term in the final lower bound.

In order to make this argument precise, it is convenient to introduce some notation similar to that used in fractal percolation, see~\cite{falcGrim}. In order to define precisely the iteration mentioned in the previous paragraph, we need to ensure that at each layer the number of subcubes is an integer. Therefore, we define $a_0=n$, and then, inductively, $a_k=a_{k-1}/\lfloor a_{k-1}/n^{p^k}\rfloor$. We also put $b_k=a_{k-1}/a_k$. Next, in order to determine the position of subcubes in the $k$th layer, we use the index set 
$$J_k=\{(i_1,\ldots, i_k)\in (\Z^d)^k:\, i_j\in\{0,\ldots, b_j-1\}^d\text{ for all }j\in\{1,\ldots,k\}\}.$$
For $k\ge1$ and $I=(i_1,\ldots,i_k)\in J_k$, we define the site $z_I=a_1i_1+\cdots+a_ki_k$ and the cube 
$$Q_{I}=(-n/2,\ldots,-n/2)+z_{I}+[0,1]^da_k,$$ 
which is also called a $k$\emph{-cube}. Note that we think of the $Q_I$ as being embedded in the torus $\T_n$ \sot it is possible that $Q_I\cap Q_{I^\p}\ne\es$, even if the $\d_\infty$-distance between $z_I$ and $z_{I^\p}$ is strictly larger than $a_k$. We say that $I,I^\p\in J_k$ are $*$-connected if \te $I_1=I,\ldots,I_m=I^\p$ \st $Q_{I_j}\cap Q_{I_{j+1}}\ne\es$ \fa $j\in\{1,\ldots,m-1\}$. In the following, we frequently consider certain neighborhoods of cubes of the form $Q_I$ for some $I\in J_k$. To be more precise, for $r>0$ we denote by $Q_I^r=\{\xi\in\T_n: \d^{\T_n}_\infty(\xi,Q_I)\le r\}$ the subset of all $\xi\in\T_n$ \st the toroidal $\d_\infty$-distance from $\xi$ to the cube $Q_I$ is at most $r$.

Next, we need to capture the property that a path $\gamma$ in $G(X^{(n)})$ that starts in a cube $Q_I$ needs a large number of hops to move far away from this cube. To be more precise, for $\varepsilon>0$ and $k\ge1$ we introduce the notion of $(\varepsilon,k)$-good indices. If $I\in J_k$, then $I$ is always $(\varepsilon,k)$-good. \Fm, inductively, if $I\in J_{k^\p}$ is \st $0\le k^\p<k$, then we say that $I$ is $(\varepsilon,k)$-good if a) $X^{(n)}\cap (Q_I\times[a_{k^\p+1},a_{k^\p}])=\es$ and b) for every $*$-connected subset $\gamma\subset J_{k^\p+1}$ that is contained in $Q^{a_{k^\p}}_I$ and consists of at least $b_{k^\p+1}/4$ elements, it holds that $\gamma$ contains at most $\varepsilon\#\gamma$ elements that are $(\varepsilon,k)$-bad. If $I=\es$, then we additionally assume that $X^{(n)}\cap (Q_I\times[n,\infty))=\es$. Sometimes, we also say that the cube $Q_I$ is $(\varepsilon,k)$-good if the index $I$ has this property. Note that it would be more intuitive if condition a) required that $X^{(n)}\cap (Q_I\times[a_{k^\p+1},\infty))=\es$. However, the present definition has the advantage that the $(\varepsilon,k)$-goodness of an index $I\in J_{k^\p}$ only depends on $X^{(n)}\cap (\T_n\times[0,a_{k^\p}])$. This property will be helpful in Lemma~\ref{kGoodProbLem} below, where we establish a stochastic domination between the configuration of $(\varepsilon,k)$-good cubes and Bernoulli site percolation.

In the following, it will also be convenient to strengthen the notion of $(\varepsilon,k)$-good cubes in order to have some control over cubes in a suitable environment of a given one. To be more precise, for $u\ge0$ and $k^\p\le k$, we say that $I\in J_{k^\p}$ is $(u,\varepsilon,k)$-good if $I^\p$ is $(\varepsilon,k)$-good \fa $I^\p\in J_{k^\p}$ \st $Q_{I^\p}\subset Q^{ua_{k^\p}}_{I}$. 

Next, in order to analyze a given self-avoiding path $\gamma=(x_1,\ldots,x_m)$ in $G(X^{(n)})$ at different scales, it is useful to introduce certain discretizations of $\gamma$. First, for any $k\ge1$ we define a function $\mu_k:\T_n\times[0,\infty)\to J_k$, where $\mu_k(\xi,r)$ denotes the uniquely determined element $I\in J_k$ satisfying $\xi\in Q_I$. 
Note that in general applying $\mu_k$ to each element of $\gamma$ results in a path that is no longer self-avoiding.
A popular technique for transforming arbitrary paths into self-avoiding ones is Lawler's method of loop erasure~\cite{saw}. Unfortunately, when performing loop erasure naively for discretizations at different scales,
the resulting self-avoiding paths may be quite incomparable with respect to moving from one scale to another. Therefore, we consider a refinement of the standard loop-erasure method, which is adapted to dealing with different scales. For $k\ge1$ define an ordered subset $\gamma^{(k,k,\LE)}$ of $\gamma$, the \emph{$(k,k)$-loop erasure of $\gamma$}, which can be identified with the standard loop erasure of the discretization of $\gamma$ via $\mu_k$. To be more precise, let $j\in\{1,\ldots,m\}$ be the largest integer \st $\mu_k(x_{j-1})=\mu_k(x_1)$. Then define recursively $\gamma^{(k,k,\LE)}=\big(x_1,\gamma_{\mathsf{t}}^{(k,k,\LE)}\big)$, where $\gamma_{\mathsf{t}}=(x_{j},\ldots,x_m)$ is the subpath of $\gamma$ starting from $x_{j}$. Next, suppose that $k^\p<k$ and let $\gamma^{(k^\p+1,k,\LE)}=(x_{m_1},\ldots,x_{m_N})$  be the $(k^\p+1,k)$-loop erasure of $\gamma$. \Ip, $m_1=1$. Choose $j\in\{1,\ldots,N\}$ as the largest integer \st $\mu_{k^\p}(x_{m_{j-1}})=\mu_{k^\p}(x_1)$. Then define recursively $\gamma^{(k^\p,k,\LE)}=\big(x_1,\gamma_{\mathsf{t}}^{(k^\p,k,\LE)}\big)$, where $\gamma_{\mathsf{t}}=(x_{m_j},\ldots,x_m)$ is the subpath of $\gamma$ starting from $x_{m_j}$. The construction of $\gamma^{(k^\p,k,\LE)}$ is illustrated in Figure~\ref{gamkLEFig}.

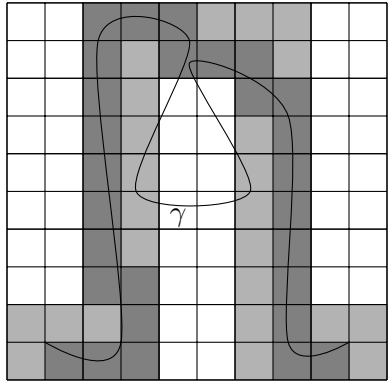
\begin{figure}[!htpb]
\centering
\begin{tikzpicture}[scale=1.0]
\fill[black!30!white] (0,0) rectangle (1,1);
\fill[black!30!white] (1,0) rectangle (2,5);
\fill[black!30!white] (2,4) rectangle (4,5);
\fill[black!30!white] (3,4) rectangle (4,0);
\fill[black!30!white] (4,0) rectangle (5,1);

\fill[black!50!white] (0.5,0) rectangle (1,0.5);
\fill[black!50!white] (1,0) rectangle (1.5,0.5);
\fill[black!50!white] (1.5,0) rectangle (2,1.5);
\fill[black!50!white] (1,1) rectangle (1.5,5.0);
\fill[black!50!white] (1.5,4.5) rectangle (2,5.0);
\fill[black!50!white] (2.0,4.5) rectangle (2.5,5.0);
\fill[black!50!white] (2.0,4) rectangle (2.5,4.5);
\fill[black!50!white] (2.5,4) rectangle (3.0,4.5);
\fill[black!50!white] (3.0,4) rectangle (3.5,4.5);
\fill[black!50!white] (3.0,3.5) rectangle (3.5,4.0);
\fill[black!50!white] (3.5,3.5) rectangle (4.0,4.0);
\fill[black!50!white] (3.5,3.0) rectangle (4.0,3.5);
\fill[black!50!white] (3.5,3.0) rectangle (4.0,0.0);
\fill[black!50!white] (4.0,0.0) rectangle (4.5,0.5);

\coordinate[label=180:$\gamma$] (u) at (2.5,2.15);
\draw plot [smooth,thick,tension=0.5] coordinates {(0.5,0.5) (1.5,0.5) (1.2,4.5) (2.4,4.5) (1.7,2.5) (3.2,2.5) (2.4,4.2) (3.7,3.5) (3.7,0.5) (4.5,0.5)};
\draw[step=1cm] (0,0) grid (5,5);
\draw[step=0.5cm] (0,0) grid (5,5);
\end{tikzpicture}
\caption{Construction of $\gamma^{(k,k,\LE)}$ (dark gray) and $\gamma^{(k-1,k,\LE)}$ (union of light and dark gray)}\label{gamkLEFig}
\end{figure}

Next, we note that if $k^\p\le k$ and $\gamma$ are \st $\gamma^{(k^\p,k,\LE)}$ hits only $(\varepsilon,k)$-good $(k^\p-1)$-cubes, then a large proportion of $k^\p$-cubes in $\gamma^{(k^\p,k,\LE)}$ are $(\varepsilon,k)$-good. 

\begin{lemma}
\label{thinBoundLem}
Let $k^\p\le k$ and $\gamma$ be a path in $G(X^{(n)}\cap(\T_n\times[0,a_{k^\p}]))$ hitting only $(\varepsilon,k)$-good $(k^\p-1)$-cubes. \Fm, assume that $\#\gamma^{(k^\p,k,\LE)}\ge b_{k^\p}/4$.
Then, the number of $(3,\varepsilon,k)$-bad $k^\p$-cubes hit by $\gamma^{(k^\p,k,\LE)}$ is at most $147^d\varepsilon \#\gamma^{(k^\p,k,\LE)}$.
\end{lemma}
\begin{proof}
Let $\gamma^{(k^\p,k,\LE)}=(x_{m_1},\ldots,x_{m_N})$. We denote by $\gamma^+$ the family of all $I\in J_{k^\p}$ \st $Q_I\subset Q^{{3a_{k^\p}}}_{\mu_{k^\p}(x_{m_j})}$ for some $j\in\{1,\ldots,N\}$. In other words, $\gamma^+$ is obtained by a suitable dilation from the discretization of $\gamma^{(k^\p,k,\LE)}$. Then the number of $(3,\varepsilon,k)$-bad $k^\p$-cubes hit by $\gamma^{(k^\p,k,\LE)}$ is at most $7^d$ times the number of $(\varepsilon,k)$-bad  $k^\p$-cubes in $\gamma^+$. \Fm, for each $j\in\{1,\ldots,N\}$ denote by $\gamma_j$ the $*$-connected component of $\gamma^+\cap Q^{a_{k^\p-1}}_{\mu_{k^\p-1}(x_{m_j})}$ containing $\mu_{k^\p}(x_{m_j})$. Since $\gamma$ hits only $(\varepsilon,k)$-good $(k^\p-1)$-cubes, we conclude that for every $j\in\{1,\ldots,N\}$, the number of $(\varepsilon,k)$-bad $k^\p$-cubes in $\gamma_j$ is at most $\varepsilon\#\gamma_j$. \Mo, note that for every $j_1,j_2\in \{1,\ldots,N\}$ the components $\gamma_{j_1}$ and $\gamma_{j_2}$ either coincide or are disjoint. In the following, we fix a subset $S\subset \{1,\ldots,N\}$ with the property that a) for every $j\in\{1,\ldots,N\}$ \tes $s\in S$ \st $\gamma_j=\gamma_s$ and b) if $s,s^\p\in S$ are \st $\gamma_s=\gamma_{s^\p}$, then $s=s^\p$. Since the union $\cup_{s\in S}\gamma_s$ covers $\gamma^+$, the number of $(\varepsilon,k)$-bad $k^\p$-cubes in $\gamma^+$ is at most $\sum_{s\in S}\varepsilon\#\gamma_s$. Finally, noting that for each $I\in\gamma^+$ \te at most $3^d$ elements $s\in S$ with $I\in\gamma_s$, we obtain that 
$$\sum_{s\in S}\varepsilon\#\gamma_s\le 3^d\varepsilon \#\gamma^+\le 21^d\varepsilon \#\gamma^{(k^\p,k,\LE)}.$$
This concludes the proof.
\end{proof}

In the following, we use the discretizations $\{\gamma^{(k^\p,k,\LE)}\}_{k^\p\in\{1,\ldots,k\}}$ to derive suitable accurate lower bounds on the number of hops in $\gamma$. The most immediate approach to do this would start from counting the number elements in $\gamma^{(k^\p,k,\LE)}$ and multiplying this number by a suitable factor reflecting the scaling in the $k$th layer.
However, if a $k^\p$-cube is hit by $\gamma^{(k^\p,k,\LE)}$, then this provides only very little information as to how many $(k^\p+1)$-cubes are hit by $\gamma^{(k^\p+1,k,\LE)}$. Another approach could be to measure the Euclidean distance between the endpoints and multiply it by a suitable factor taking into account the scale of discretization. However, also this idea is problematic, since it is not clear how to determine an upper bound for the number of $(\varepsilon,k)$-bad cubes occurring in a path in terms of the distance of the endpoints. Therefore, in order to measure the length of a discretized path, we propose a slightly more refined approach. It is adapted to changing scales, and, moreover, the length still grows at least linearly in the number of elements of a path.

Let $k\ge1$ and $\gamma=(x_1,\ldots,x_m)$ be a self-avoiding path in $J_k$. Let $D=(d_1,d_2,\ldots,d_{m^\p})$ be an ordered subset of $\{2,\ldots,m\}$. We say that $D$ forms an \emph{independent subset} of $\gamma$ if $Q_{I_{d_i-1}}\cap Q_{I_{d_{i+1}}}=\es$ \fa $i\in\{1,\ldots,m^\p-1\}$. By $\lambda(\gamma)$ we denote the \emph{length} of $\gamma$, which is defined as $m^\p_{\mathsf{max}}-1$, where $m^\p_{\mathsf{max}}$ is the maximal size of an independent subset of $\gamma$. At first sight, it might seem unnatural to require $Q_{I_{d_i-1}}\cap Q_{I_{d_{i+1}}}=\es$ instead of $Q_{I_{d_i}}\cap Q_{I_{d_{i+1}}}=\es$. However, when considering a linear arrangement of $m\ge2$ adjacent elements of $J_k$, then the second possibility would lead to a length of at most $\lceil m/2\rceil$, whereas the first yields the more accurate value $m-2$.

In the following, we often consider the case, where $k^\p\le k$ and $\gamma=(x_1,\ldots,x_m)$ is a self-avoiding path in $G(X^{(n)}\cap(\T_n\times[0,a_{k^\p}]))$. Then, an independent subset of $\mu_{k^\p}(\gamma^{(k^\p,k,\LE)})$ is also called a $(k^\p,k)$-\emph{independent subset} of $\gamma$ and we write $\lambda_{k^\p,k}(\gamma)$ for $\lambda(\mu_{k^\p}(\gamma^{(k^\p,k,\LE)}))$, which is called the $(k^\p,k)$-\emph{length} of $\gamma$. Next, we provide an affine lower bound for $(k^\p,k)$-lengths in terms of number of sites.

\begin{lemma}
\label{lengthLowBoundLem}
Let $k^\p\le k$ and $\gamma$ be a self-avoiding path in $G(X^{(n)}\cap(\T_n\times[0,a_{k^\p}]))$. Then,
$$\lambda_{k^\p,k}(\gamma)+1\ge10^{-d}\#\gamma^{(k^\p,k,\LE)}.$$
\end{lemma}
\begin{proof}
The proof proceeds by induction on $\#\gamma$. If $\lambda_{k^\p,k}(\gamma)=0$, then \tes $I\in J_{k^\p}$ \st $\gamma$ remains inside $Q^{a_{k^\p}}_I$. \Ip, $\#\gamma^{(k^\p,k,\LE)}\le3^d$. Otherwise, write $\gamma^{(k^\p,k,\LE)}=(x_{m_1},\ldots, x_{m_N})$ and let $D=(d_1,\ldots,d_{m^\p})$ be an ordered subset of $\{1,\ldots, N\}$ \st the ordered set $(\mu_{k^\p}(x_{m_{d_1}}),\ldots,\mu_{k^\p}(x_{m_{d_{m^\p}}}))$ corresponds to a $(k^\p,k)$-independent subset of $\gamma$ of size $m^\p=\lambda_{k^\p,k}(\gamma)+1\ge2$. \Fm, we also assume that among all $(k^\p,k)$-independent subsets of maximal size, $D$ is chosen as the lexicographic maximum. Next, we claim that $d_2-d_1\le 9^d+1$. Otherwise, we could choose $d^\p\in\{d_1+1,\ldots,d_2-1\}$ \st $Q_{\mu_{k^\p}(x_{m_{d^\p}})}\cap Q_{\mu_{k^\p}(x_{m_{d_1-1}})}=\es$ and $Q_{\mu_{k^\p}(x_{m_{d^\p-1}})}\cap Q_{\mu_{k^\p}(x_{m_{d_2}})}=\es$, contradicting the maximality property used to define $D$. Similarly, we see that $d_1\le 3^d+1$. Therefore, $\#\gamma_{\mathsf{t}}^{(k,k^\p,\LE)}\ge\#\gamma^{(k,k^\p,\LE)}-10^d$, where $\gamma_{\mathsf{t}}=(x_{m_{d_2-1}},\ldots,x_{m})$ denotes the subpath of $\gamma$ starting at $x_{m_{d_2-1}}$. \Fm, since $D$ was chosen as the lexicographic maximum of all independent subsets of maximal size, we also see that $\lambda_{k^\p,k}(\gamma)>\lambda_{k^\p,k}(\gamma_{\mathsf{t}})$. Hence,
\begin{align*}
\lambda_{k^\p,k}(\gamma)+1\ge\lambda_{k^\p,k}(\gamma_{\mathsf{t}})+2\ge 1+10^{-d}\#\gamma_{\mathsf{t}}^{(k^\p,k,\LE)}\ge10^{-d}\#\gamma^{(k^\p,k,\LE)}.
\end{align*}
\end{proof}

Next, we show that $(k^\p,k)$-lengths exhibit a good behavior with respect to changes in scale.

\begin{lemma}
\label{lvlChangeLem}
Let $2\le k^\p\le k$ and $\gamma$ be a path in $G(X^{(n)}\cap(\T_n\times[0,a_{k^\p}]))$. Then, 
$$\lambda_{k^\p,k}(\gamma)\ge\lambda_{k^\p-1,k}(\gamma)b_{k^\p}.$$
\end{lemma}
\begin{proof}
Let $\gamma=(x_1,\ldots,x_m)$.
We proceed by induction on $\#\gamma$, noting that the case $\lambda_{k^\p-1,k}(\gamma)=0$ is trivial. Next, let $\gamma^{(k^\p,k,\LE)}=(x_{m_1},\ldots,x_{m_N})$ and choose $u,v\in \{1,\ldots,N\}$ \st $\mu_{k^\p-1}(x_{m_{u}})$ and $\mu_{k^\p-1}(x_{m_{v}})$ are the first two elements in a maximal $(k^\p-1,k)$-independent subset of $\gamma$. Putting $I_1=\mu_{k^\p-1}(x_{m_{u-1}})$ and $I_2=\mu_{k^\p-1}(x_{m_{v}})$, we note that $Q_{I_1}\cap Q_{I_2}=\es$.
\Ip, $Q_{\mu_{k^\p}(x_{m_{u-1}})}\not\subset Q_{\mu_{k^\p}(x_{m_{v}})}^{a_{k^\p-1}}$.
Hence, denoting by $\gamma_{\mathsf{t}}=(x_{m_{v-1}},\ldots,x_{m})$ the subpath of $\gamma$ starting from $x_{m_{v-1}}$, we obtain that 
\begin{align*}
\lambda_{k^\p,k}(\gamma)\ge b_{k^\p}+\lambda_{k^\p,k}(\gamma_{\mathsf{t}})\ge(1+\lambda_{k^\p-1,k}(\gamma_{\mathsf{t}}))b_{k^\p}\ge \lambda_{k^\p-1,k}(\gamma)b_{k^\p},
\end{align*}
as claimed.
\end{proof}

Let $\gamma=(x_1,\ldots,x_m)=((\xi_1,r_1),\ldots,(\xi_m,r_m))$ be a self-avoiding path in $G(X^{(n)}\cap(\T_n\times[0,a_{k^\p}]))$ for some $k^\p<k$. In order to derive helpful lower bounds on $m$, we make use of the observation that inside any $(\varepsilon,k)$-good cube $Q_I$ with $I\in J_{k^\p}$ the path $\gamma$ consists of segments of length at most $a_{k^\p+1}$, and, moreover, most of the $(k^\p+1)$-cubes that hit this path are also $(\varepsilon,k)$-good. Hence, it is convenient to identify subpaths of $\gamma$ that do not intersect $(\varepsilon,k)$-bad cubes. Also the subpaths should not be too short and, finally, some care should be taken to ensure compatibility with taking $(k^\p,k)$-loop erasures. 

To be more precise, let $k^\p\le k$ and $\gamma=(x_1,\ldots,x_m)$ be a path in $G(X^{(n)}\cap(\T_n\times[0,a_{k^\p}]))$. \Fm, let $A\subset\mu_{k^\p}(\gamma^{(k^\p,k,\LE)})$ and write $\gamma^{(k^\p,k,\LE)}=(x_{m_1},\ldots,x_{m_N})$. Then, we define a family of disjoint subpaths $\Phi_A(\gamma)=\{\gamma_i\}_{i\in\{1,\ldots,\ell\}}$ inductively as follows.
\begin{itemize} 
\item If $\mu_{k'}(x_1)\not\in A$, then $\Phi_A(\gamma)=\Phi_{A\setminus \{\mu_{k^\p}(x_1)\}}(x_{m_2},\ldots,x_m)$.
In other words, for the construction of $\Phi_A(\gamma)$ we discard the initial segment of $\gamma^{(k^\p,k,\LE)}$ not belonging to $A$.
\item Otherwise, let $L\in\{2,\ldots,N\}$ be the smallest integer \st $\mu_{k^\p}(x_{m_L})\in A$.
\begin{itemize}
\item If for every $j\in\{1,\ldots,m_L\}$ \tes $I\in A$ \st $\xi_{j}\in Q^{3a_{k^\p}}_I$, then put $\ell=\ell^\p$, $\gamma_1=\mathsf{concat}((x_1,\ldots,x_{m_{L}}),\gamma_1^\p)$ and $\gamma_i=\gamma_i^\p$ for $i\in\{2,\ldots,\ell^\p\}$, where $\Phi_{A\setminus\{\mu_{k^\p}(x_1)\}}((x_{m_L},\ldots,x_m))=\{\gamma_i^\p\}_{i\in\{1,\ldots,\ell^\p\}}$, and where $\mathsf{concat}$ denotes concatenation of paths. In other words, if until reaching $x_{m_L}$ the path $\gamma$ stays close to $A$, then to construct $\gamma_1$, we proceed recursively by taking the first subpath from $\Phi_{A\setminus\{\mu_{k^\p}(x_1)\}}((x_{m_L},\ldots,x_m))=\{\gamma_i^\p\}_{i\in\{1,\ldots,\ell^\p\}}$ and pasting it to the subpath from $x_1$ to $x_{m_L}$. 
\item Otherwise, let $j\in\{1,\ldots,m_L\}$ be the smallest integer \st there does not exist $I\in A$ with $\xi_{j}\in Q^{3a_{k}}_I$ and put 
\MLine{\Phi_A(\gamma)=\{(x_1,\ldots,x_{j-1})\}\cup\Phi_{A\setminus\{\mu_{k^\p}(x_{1})\}}((x_{m_L},\ldots,x_m)).}
In other words, $\gamma_1$ is the longest initial segment of $\gamma$ that stays close to $A$; the other subpaths are constructed inductively.
\end{itemize}
\end{itemize}
The construction of $\Phi_A(\gamma)$ is illustrated in Figure~\ref{phiAFig}.

\begin{figure}[!htpb]
\begin{subfigure}[b]{0.52\textwidth}
\centering
\begin{tikzpicture}[scale=1.0]

\fill[black!30!white] (0,0) rectangle (1,1);
\fill[black!30!white] (1,0) rectangle (2,1);
\fill[black!30!white] (2,0) rectangle (3,1);
\fill[black!30!white] (3,0) rectangle (4,1);
\fill[black!30!white] (4,0) rectangle (5,1);

\coordinate[label=180:$\gamma$] (u) at (2.7,2.5);

\draw plot [smooth,thick,tension=0.5] coordinates {(0.5,0.5) (1.5,0.5) (0.9,1.5) (1.2,4.5) (1.9,0.5) (2.7,2.5) (3.2,2.5) (3.2,4.5) (3.7,3.5) (3.7,0.5) (4.5,0.5)};
\draw[step=1cm] (0,0) grid (5,5);
\end{tikzpicture}
\caption{Path $\gamma$ (solid line) and set $A$ (shaded region)}
\end{subfigure}
        \begin{subfigure}[b]{0.47\textwidth}
\centering
\begin{tikzpicture}[scale=1.0]
\fill[black!30!white] (0,0) rectangle (1,1);
\fill[black!30!white] (1,0) rectangle (2,1);
\fill[black!30!white] (2,0) rectangle (3,1);
\fill[black!30!white] (3,0) rectangle (4,1);
\fill[black!30!white] (4,0) rectangle (5,1);
\draw plot [smooth,thick,tension=0.5] coordinates {(0.5,0.5) (1.5,0.5) (0.9,1.5) (1.2,4.5) (1.9,0.5) (2.7,2.5) (3.2,2.5) (3.2,4.5) (3.7,3.5) (3.7,0.5) (4.5,0.5)};
\fill[white] (0,4) rectangle (5,5);
\fill[white] (1.1,3) rectangle (2,4);
\fill[white] (1.5,1) rectangle (2,3);
\fill[white] (1.3,2) rectangle (2,3);
\fill[white] (3.5,1) rectangle (4,4);
\fill[black!30!white] (1.5,0) rectangle (2,1);

\coordinate[label=180:$\gamma_1$] (u) at (0.9,1.5);
\coordinate[label=180:$\gamma_2$] (u) at (2.5,1.7);
\coordinate[label=180:$\gamma_3$] (u) at (3.7,0.5);

\draw[step=1cm] (0,0) grid (5,5);
\end{tikzpicture}
\caption{Subpaths $\gamma_1$, $\gamma_2$ and $\gamma_3$}
\end{subfigure}
\caption{Construction of $\Phi_A(\gamma)$}\label{phiAFig}
\end{figure}

Now, we can use the construction $\Phi_A$ to obtain $(\varepsilon,k)$-good cubes at smaller scales.

\begin{lemma}
\label{subDivLem}
Let $\gamma$ be a path in $G(X^{(n)}\cap(\T_n\times[0,a_{k^\p}]))$. Denote by $\{\gamma_i\}_{i=1}^\ell$ the collection of subpaths of $\gamma$ obtained by applying the construction $\Phi_A$ with the family $A$ as the set of those $(3,\varepsilon,k)$-good indices $I\in J_{k^\p}$ \st $\gamma^{(k^\p,k,\LE)}$ hits $Q_I$ and $Q^{3a_{k^\p}}_I$ does not contain the endpoint of $\gamma$. Then, $A\subset\cup_{i=1}^{\ell}\mu_{k^\p}\big(\gamma_i^{(k^\p,k,\LE)}\big)$, and for every $i\in\{1,\ldots,\ell\}$,
\begin{enumerate}
\item[1.] the path $\gamma_i$ hits only $(\varepsilon,k)$-good $k^\p$-cubes; in particular, $\gamma_i$ is a path in $G(X^{(n)}\cap(\T_n\times[0,a_{k^\p+1}]))$, 
\item[2.] the path $\gamma_i$ starts in $Q_I$ for some $I\in A$ and ends in $Q_{I^\p}$ for some $I^\p$ with $Q_{I^\p}\cap Q_{I^\pp}=\es$ \fa $I^\pp\in A$; in particular, $\#\gamma_i^{(k^\p+1,k,\LE)}\ge b_{k^\p+1}/4$, and
\item[3.] $\sum_{i=1}^\ell\lambda_{k^\p,k}(\gamma_i)\ge\lambda_{k^\p,k}(\gamma)-\big(\#\gamma^{(k^\p,k,\LE)}-\#A\big)$.
\end{enumerate}
\end{lemma}
\begin{proof}
The relation $A\subset\cup_{i=1}^{\ell}\mu_{k^\p}\big(\gamma_i^{(k^\p,k,\LE)}\big)$ and claim 1 follow immediately from the definition of $\Phi_A$. Claim 2 is shown by induction on $\#\gamma$, noting the assertion is trivial if $\ell=0$. 
Next, we deal with claim 2 if $\ell>0$. Write $\gamma^{(k^\p,k,\LE)}=(x_{m_1},\ldots,x_{m_N})$. Clearly, our attention can be restricted to the case, where $\mu_{k^\p}(x_1)\in A$. Let $L\in\{2,\ldots,N\}$ be the smallest integer \st $\mu_{k^\p}(x_{m_L})\in A$. If for every $j\in\{1,\ldots,m_L\}$ \tes $I\in A$ \st $\xi_{j}\in Q^{3a_{k^\p}}_I$, then the claim follows immediately from the induction hypothesis. Otherwise, let $j\in\{1,\ldots,m_L\}$ be the smallest integer \st there does not exist $I\in A$ with $\xi_{j}\in Q^{3a_{k^\p}}_I$. Since $\gamma$ is a path in $G(X^{(n)}\cap(\T_n\times[0,a_{k^\p}]))$, we conclude that $Q_{\mu_{k^\p}(x_1)}\cap Q_{\mu_{k^\p}(x_{j-1})}=\es$. This completes the proof of claim 2 for $\gamma_1$, whereas for $\gamma_i$ with $i\ge2$, we may conclude by induction.

It remains to prove claim 3. Choose an ordered subset $(d_1,\ldots,d_{m^\p})$ of $\{1,\ldots,N\}$ corresponding to a maximal $(k^\p,k)$-independent subset of $\gamma$. \Fm, for $i\in\{1,\ldots,\ell\}$ let $D_i$ be the subset consisting of all $d^\p\in D$ \st $\mu_{k^\p}(x_{m_{d^\p}})\in A$ and $x_{m_{d^\p}}\in \gamma_i$. Note that if we remove from $D_i$ the element corresponding to the starting point of $\gamma_i$, then we obtain a $(k^\p,k)$-independent subset of $\gamma_i$. \Mo, we assert that this independent subset can be enlarged by two further elements. Once this assertion is shown, we see that $\lambda_{k^\p,k}(\gamma_i)\ge \#D_i$ and summing over all $i\in\{1,\ldots,\ell\}$ completes the proof. In order to prove the assertion, we choose $d_{i,1}$ as the largest index from $D_i$ \st $\xi_{m_{d_{i,1}}}\in Q^{a_{k^\p}}_I$ for some $I\in A$. Similarly, we choose $d_{i,2}$ as the largest index \st $\xi_{m_{d_{i,2}}}\in Q^{2a_{k^\p}}_I$ for some $I\in A$. Then, we can enlarge the independent subset $D_i$ by adding $d_{i,1}+1$ and $d_{i,2}+1$. Indeed, for any $d^\p\in D_i$ we have $\mu_{k^\p}(x_{m_{d^\p}})\in A$, whereas there is no $I\in A$ \st $Q_{\mu_{k^\p}(x_{m_{d_{i,1}+1}})}\subset Q^{a_{k^\p}}_I$. Similarly, $Q_{\mu_{k^\p}(x_{m_{d_{i,1}}})}\subset Q^{a_{k^\p}}_I$ for some $I\in A$, whereas there is no $I\in A$ \st $Q_{\mu_{k^\p}(x_{m_{d_{i,2}+1}})}\subset Q_I^{2a_{k^\p}}$. Therefore, $Q_{\mu_{k^\p}(x_{m_{d_{i,1}}})}\cap Q_{\mu_{k^\p}(x_{m_{d_{i,2}+1}})}=\es$, and this completes the proof.
\end{proof}

Combining the previous auxiliary results, we now provide a lower bound for $\#\gamma$.
\begin{lemma}
\label{indChemLem}
Let $k^\p\le k$ and $\gamma$ be a path in $G(X^{(n)}\cap(\T_n\times[0,a_{k^\p}]))$ hitting only $(\varepsilon,k)$-good $(k^\p-1)$-cubes. \Fm, assume that $\#\gamma^{(k^\p,k,\LE)}\ge b_{k^\p}/4$. Then \tes a constant $c>0$ \st if $b_{k}\ge c$, then $\#\gamma\ge(1-1500^d\varepsilon)^{k-k^\p}\lambda_{k^\p,k}(\gamma)\prod_{j=k^\p+1}^kb_j$.
\end{lemma}
\begin{proof}
The proof proceeds via backward induction on $k^\p$, the case $k^\p=k$ being trivial. First, let $A\subset J_{k^\p}$ be as in Lemma~\ref{subDivLem} and write $\Phi_A(\gamma)=\{\gamma_i\}_{i=1}^\ell$. By induction hypothesis and Lemma~\ref{lvlChangeLem}, we conclude that for every $i\in\{1,\ldots,\ell\}$,
\begin{align*}
\#\gamma_i&\ge(1-1500^d\varepsilon)^{k-k^\p-1}\lambda_{k^\p+1,k}(\gamma_i)\prod_{j=k^\p+2}^kb_j\ge(1-1500^d\varepsilon)^{k-k^\p-1}\lambda_{k^\p,k}(\gamma_i)\prod_{j=k^\p+1}^kb_j.
\end{align*}
By Lemma~\ref{subDivLem}, we have $A\subset\cup_{i=1}^{m^\p}\mu_{k^\p}\big(\gamma_i^{(k^\p,k,\LE)}\big)$ and by Lemma~\ref{thinBoundLem} the number of $(3,\varepsilon,k)$-bad cubes hit by $\gamma^{(k^\p,k,\LE)}$ is at most $147^d\varepsilon\#\gamma^{(k^\p,k,\LE)}$. Hence, if $b_k$ is sufficiently large, then part 3 of Lemma~\ref{subDivLem}, Lemma~\ref{thinBoundLem} and Lemma~\ref{lengthLowBoundLem} imply that 
$$\sum_{i=1}^{\ell}\lambda_{k^\p,k}(\gamma_i)\ge\lambda_{k^\p,k}(\gamma)-147^d\varepsilon\#\gamma^{(k^\p,k,\LE)}-7^d\ge(1-1500^d\varepsilon)\lambda_{k^\p,k}(\gamma),$$
\sot
\begin{align*}
\sum_{i=1}^{\ell}\#\gamma_i\ge (1-1500^d\varepsilon)^{k-k^\p-1}\sum_{i=1}^{\ell}\lambda_{k^\p,k}(\gamma_i)\prod_{j=k^\p+1}^kb_j\ge(1-1500^d\varepsilon)^{k-k^\p}\lambda_{k^\p,k}(\gamma)\prod_{j=k^\p+1}^kb_j.
\end{align*}
This completes the proof.
\end{proof}
Next, we show that the process of $(3,\varepsilon,k)$-good indices dominates a Bernoulli site percolation process with arbitrarily high marginal probability. This uses similar arguments as in~\cite[Lemma 2.2]{mensh}.

\begin{lemma}
\label{kGoodProbLem}
Let $\varepsilon>0$ and $\rho\in(0,1)$ be arbitrary. \Fm, let $k=k(n)$ be \st $n^{-p^{k(n)}}\in o(1)$. Then \tes $n_0\ge1$ \st if $n\ge n_0$ and $k^\p\in\{0,\ldots,k\}$, then the family of $(\varepsilon,k)$-good cubes in $J_{k^\p}$ stochastically dominates a Bernoulli site process on $J_{k^\p}$ with marginal probability $\rho$.
\end{lemma}

%
\begin{proof}
In order to prove the claim, it is convenient to separate clearly the two conditions used in the definition of $(\varepsilon,k)$-goodness. First, an index $I\in J_{k^\p}$ is said to be \emph{short-ranged} if $X^{(n)}\cap (\T_n\times[a_{k^\p+1},a_{k^\p}])=\es$. Second, $I\in J_{k^\p}$ is said to be \emph{iterable} if for every $*$-connected subset $\gamma\subset J_{k^\p+1}$ that is contained in $Q^{a_{k^\p}}_I$ and is of size at least $b_{k^\p+1}/4$, it holds that $\gamma$ contains at most $\varepsilon\#\gamma$ indices that are $(\varepsilon,k)$-bad. By construction, the configuration of short-ranged indices in $J_{k^\p}$ is independent of the configuration of iterable indices in $J_{k^\p}$. Hence, to obtain the desired stochastic domination of a Bernoulli site process, it suffices to consider the two types of configurations separately.

To be more precise, let $\varepsilon>0$ and $\rho\in(0,1)$ be arbitrary. \Fm, put $q=2^{-2\cdot 3^d/\varepsilon}$. We prove that there exist $n_0\ge1$ (depending only on $d$, $\varepsilon$ and $\rho$) with the following properties, where without loss of generality we may assume that $\rho\ge1-q$.
\begin{enumerate}
\item[1.] Let $n\ge n_0$ and $k^\p\in\{1,\ldots,k\}$. Then the process of short-ranged indices in $J_{k^\p}$ stochastically dominates a Bernoulli site process on $J_{k^\p}$ with marginal probability given by $\sqrt{\rho}$.
\item[2.] Let $n\ge n_0$ and $k^\p\in\{1,\ldots,k\}$. If the process of $(\varepsilon,k)$-good indices in $J_{k^\p+1}$ dominates a Bernoulli site process on $J_{k^\p+1}$ with marginal probability given by $1-q$, then the process of iterable indices in $J_{k^\p}$ dominates a Bernoulli site process on $J_{k^\p}$ with marginal probability given by $\sqrt{\rho}$.
\end{enumerate}
Once these two claims are shown, we conclude by induction.

For the first claim, we note that the configuration of short-ranged elements in $J_{k^\p}$ is already a Bernoulli site process. \Fm, provided that $n\ge1$ is sufficiently large, the marginal probability of failing to be short-ranged is at most 
$$2\beta a_{k^\p}^d a_{k^\p+1}^{-s}\le 2\beta n^{p^{k^\p} d-p^{k^\p}ps}=2\beta n^{-p^{k^\p}\varepsilon_1},$$
where $\varepsilon_1=ps-d>0$. \Ip, the latter expression tends to $0$ uniformly over all $k^\p\in\{1,\ldots,k(n)\}$ as $n\to\infty$.

It remains to prove the second claim. 
By assumption, the process of $(\varepsilon,k)$-good indices in $J_{k^\p+1}$ dominates a Bernoulli site process on $J_{k^\p+1}$ with marginal probability $1-q$. An element $I\in J_{k^\p+1}$ is called \emph{dom-bad} if it is a closed site in this Bernoulli site process. \Fm, we say that $I\in J_{k^\p}$ is \emph{dom-iterable} if for every $*$-connected subset $\gamma\subset J_{k^\p+1}$ that is contained in $Q^{a_{k^\p}}_I$ and consists of at least $b_{k^\p+1}/4$ elements, it holds that $\gamma$ contains at most $\varepsilon\#\gamma$ elements that are dom-bad. 

For any fixed $u\in\{1,\ldots,3^db^d_{k^\p+1}\}$, the cube $Q_I^{a_{k^\p}}$ contains at most $3^db^d_{k^\p+1}2^{(3^d-1)u}$  $*$-connected subsets consisting of precisely $u$ elements, see~\cite[Lemma 9.3]{penrose}. For any such $*$-connected set there exist at most $2^{u}$ possibilities to choose the location of dom-bad sites.
Hence, the probability that $I$ is not dom-iterable is at most 
\begin{align}
\label{kGoodEq}
3^db^d_{k^\p+1}\sum_{u\ge b_{k^\p+1}/4} 2^{(3^d-1)u}2^{u} q^{\varepsilon u}\le 3^db_{k^\p+1}^d2^{1-3^d\lceil b_{k^\p+1}/4\rceil}.
\end{align}
If $k^\p\in\{1,\ldots,k(n)\}$ and $n$ is sufficiently large, then $b_{k^\p}\ge n^{p^{k^\p}(1-p)}/2$. Therefore, we see that~\eqref{kGoodEq} tends to $0$ uniformly over all $k^\p\in\{1,\ldots,k(n)\}$ as $n\to\infty$. The proof of the second claim is concluded by invoking~\cite[Corollary 1.4]{domProd}. Alternatively, it is also possible to apply~\cite[Theorem 2.1]{mensh}, which also yields an explicit bound for the value of $n$ that is needed to achieve the desired domination.
\end{proof}

Now, we have collected all preliminary results required to complete the proof of Theorem~\ref{subCrit}. 

\begin{proof}[Proof of Theorem~\ref{subCrit}]
Define $k=k(n)=\lfloor (\log^{(4)}n-\log^{(2)}n)/\log p\rfloor$, where for $m\ge1$ we denote by $\log^{(m)}n$ the $m$-fold iterated logarithm. \Ip, 
$$n^{p^k}\in \big(\log^{(2)}n^{},(\log^{(2)}n)^{p^{-1}}\big).$$
\Mo, we conclude from Lemmas~\ref{indChemLem} and~\ref{kGoodProbLem} that whp the number of hops between $q(-n\mathsf{e}_1/4)$ and $q(n\mathsf{e}_1/4)$ is at least
\begin{align*}
(1-1500^d\varepsilon)^{k-1}\lambda_{1,k}(\gamma)\prod_{j=2}^kb_j&\ge\tfrac{1}{4}(1-1500^d\varepsilon)^{k-1}n^{1-p^k}\\
&\ge\tfrac{1}{4} (1-1500^d\varepsilon)^{k-1}n(\log^{(2)} n)^{-p^{-1}}.
\end{align*}
It thus remains to provide a suitable lower bound for $(1-1500^d\varepsilon)^k$. Indeed, 
\begin{align*}
k\log(1-1500^d\varepsilon)\ge(-2\log(1-1500^d\varepsilon)/\log p)\log^{(2)}n,
\end{align*}
\sot choosing $\varepsilon>0$ sufficiently close to $0$ to ensure that $-\log(1-1500^d\varepsilon)/\log p\ge -\alpha/4$ shows that $(1-1500^d\varepsilon)^k\ge (\log n)^{-\alpha/2}$. 
\end{proof}

\subsection{Regime $s=d$}
In the present subsection, we consider the critical regime, where $s=d$ and provide a proof of Theorem~\ref{sd1Thm}. We first explain the main ideas before presenting all the details. If \tes $x=(\xi,r)\in X^{(n)}$ with $r\ge\sqrt{d}n$, then all points of $X^{(n)}$ are connected to $x$ directly by an edge. On the other hand, if $\T_n$ is not covered by a single ball, then the torus can be subdivided into smaller subcubes, and we try to cover these subcubes by smaller balls. While some are now covered, others will stay uncovered. At this points one proceeds iteratively, subdividing the remaining subcubes into subsubcubes and aiming at covering these smaller cubes by smaller balls. We claim that this algorithm terminates with high probability, yielding a connected random geometric graph whose diameter is bounded from above by the total number of subcubes introduced in this construction.

Similar to Section~\ref{subCritReg}, in order to make these arguments rigorous, it is useful to highlight a link to fractal percolation. However, now the fractal percolation process is considerably simpler, since each occurring subcube is subdivided into precisely $2^d$ subsubcubes, irrespective of the level, in which the original subcube is located. In the present setting, the total index set $\mc{J}$ is therefore given by $\mc{J}=\cup_{m\ge0}\{0,1\}^{md}$, where the symbol $\cup$ is interpreted as disjoint union. \Fm, for each $I=(i_1,\ldots,i_m)\in\mc{J}$ we put 
$$Q_I=(-n/2,-n/2,\ldots,-n/2)+n\sum_{j=1}^m2^{-j}i_j+[0,n2^{-m}]^d.$$

%
%

\begin{proof}[Proof of Theorem~\ref{sd1Thm}]
In order to make the sketch presented at the beginning of the subsection rigorous, we introduce a fractal percolation process $\{Z(I)\}_{I\in\mc{J}}$. For $k\ge0$ and $I\in \{0,1\}^{kd}$ put $Z(I)=0$ if and only if \tes $(\xi,r)\in X$ \st $\xi\in Q_{I}$ and $r\in(\sqrt{d}2^{-k+1}n,\sqrt{d}2^{-k+2}n)$. 
The number of points $(\xi,r)\in X$ \st $\xi\in Q_{I}$ and $r\in (\sqrt{d}2^{-k+1}n,\sqrt{d}2^{-k+2}n)$ is Poisson distributed with parameter
\begin{align*}
n^d2^{-kd} (\P(R>\sqrt{d}2^{-k+1}n)-\P(R>\sqrt{d}2^{-k+2}n)).
\end{align*}
Provided that $\sqrt{d}2^{-k+1}n\ge t_0(1/2)$ this expression can be bounded from below by
\begin{align*}
n^d2^{-kd} d^{-d/2}\beta(2^{kd-d-1}n^{-d}-3\cdot2^{kd-2d-1}n^{-d})\ge \beta d^{-d/2}2^{-2d-1}.
\end{align*}
Hence, if additionally $\beta\ge d^{d/2}2^{2d+1}(d+1)\log 2$, then
\begin{align}
\label{lowCovProp}
\P(Z(I)=1)\le \exp(-\beta d^{-d/2}2^{-2d-1})\le 2^{-d-1}.
\end{align}
Since each cube $Q_I$ gives rise to $2^d$ subcubes, this already provides a first strong indication for the relationship to subcritical Galton-Watson processes.

To make this precise, it is convenient to introduce some auxiliary structures. By $C_k$, $k\ge0$ we denote the union of retained cubes at the $k$th level. That is, $C_0=[-n/2,n/2]^d$ and if $k\ge0$, then 
$$C_{k+1}=C_k\cap \bigcup_{I\in \{0,1\}^{(k+1)d}:\, Z(I)=1} Q_I.$$
\Fm, we also consider the index sets $A^{(0)}_k,A^{(1)}_k\subset\{0,1\}^{kd}$ whose associated cubes are discarded/retained at the $k$th step. That is $A^{(\sigma)}_k=\{I\in\{0,1\}^{kd}: Q_I\subset C_{k-1}\text{ and } Z(I)=\sigma\}$. Finally, we construct a backbone $B\subset X^{(n)}$ \st if $C_m=\es$ for some $m\ge1$, then a) any point of $X^{(n)}$ is connected by an edge in $G(X^{(n)})$ to some point in $B$ and b) $B$ is a connected set in $G(X^{(n)})$. To construct $B$, we choose for each $I\in A^{(0)}_k$ a point $x_I\in X^{(n)}$ \st $\xi\in Q_{I}$ and $r\in(\sqrt{d}2^{-k+1}n,\sqrt{d}2^{-k+2}n)$. Then, we denote by $B$ the collection consisting of the points $x_I$, where $I\in A^{(0)}_k$ for some $k\ge0$. See Figure~\ref{bbFig} for an illustration of the construction of $B$.

\Ip,
$$\#B=\sum_{k\ge0} \#A^{(0)}_{k+1}\le 2^d\sum_{k\ge0} \#A^{(1)}_k.$$
Now, we show that the set $B$ has the desired properties. Since $C_m=\es$, we conclude that the union of the cubes $Q_I$, with $I\in A^{(0)}_k$ for some $k\ge0$ covers $\T_n$. If $x=(\xi,r)\in X^{(n)}$ is arbitrary, then by choosing $I\in A^{(0)}_k$ \st $\xi\in Q_I$ and noting that $Q_I\subset \BT_{r_I}(\xi_I)$, we see that $x$ is connected to $x_I$ by an edge in $G(X^{(n)})$. \Mo, if $I_1\in A^{(0)}_{k_1}$ and $I_2\in A^{(0)}_{k_2}$ are \st $k_1\le k_2$ and $Q_{I_1}\cap Q_{I_2}\ne\es$, then $Q_{I_1}\cup Q_{I_2}\subset \BT_{r_{I_1}}(\xi_{I_1})$. Hence, $x_{I_1}$ and $x_{I_2}$ are connected by an edge in $G(X^{(n)})$. Since for any $I\in A^{(0)}_k$ and $I^\p\in A^{(0)}_{k^\p}$ we can find $I_1=I,\ldots, I_m=I^\p$ \st $I_j\in A^{(0)}_{k_j}$ and $Q_{I_j}\cap Q_{I_{j+1}}\ne\es$, this proves the second claim on $B$.

\begin{figure}[!htpb]
\centering
\includegraphics[width=9cm]{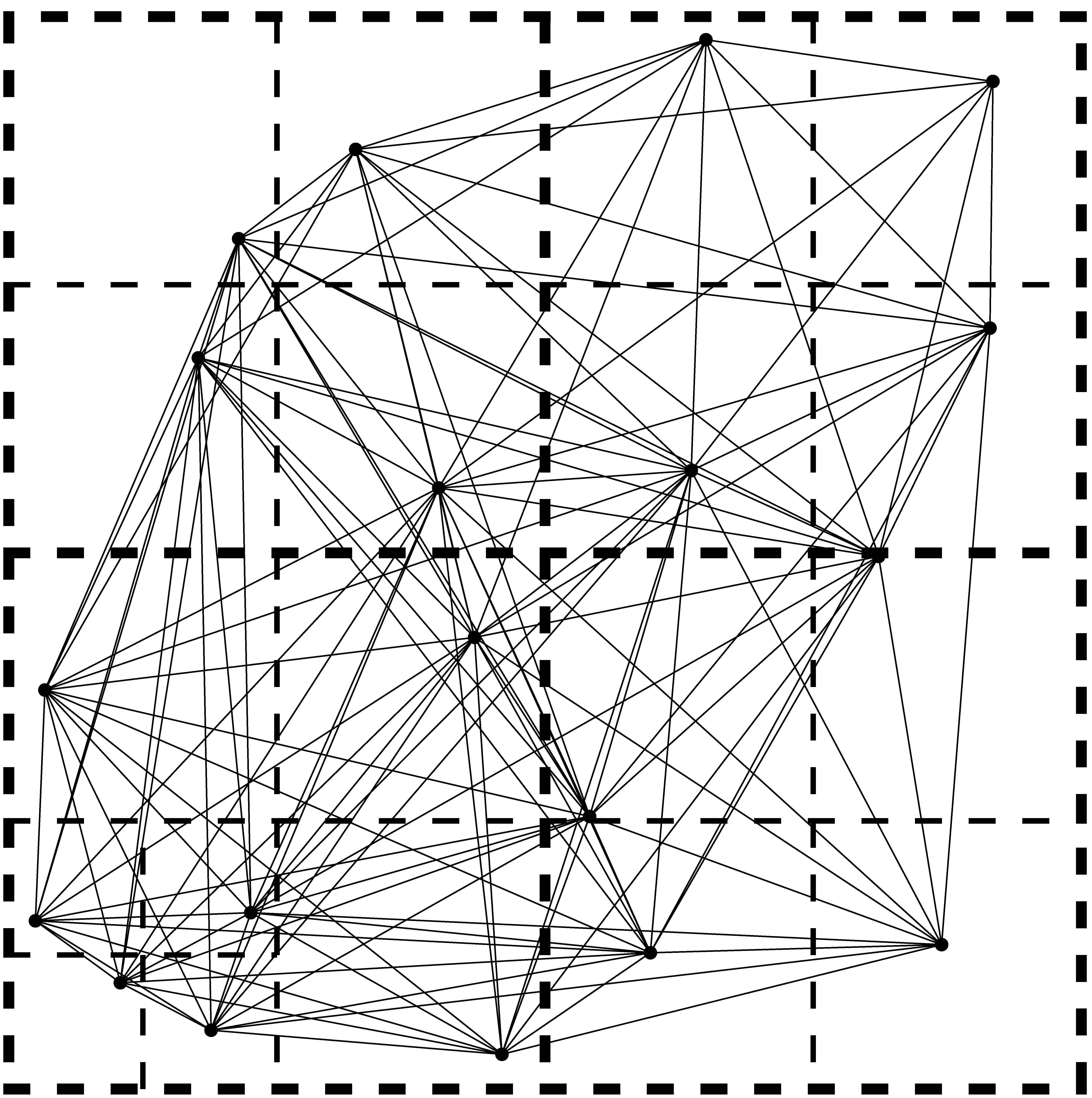}
\caption{Backbone of a scale-free Gilbert graph}
\label{bbFig}
\end{figure}

Hence, if \tes $m\ge1$ \st $C_m=\es$, then the diameter of $G(X^{(n)})$ is at most $2+\#B$. It remains to compare $\#B$ to the total progeny of a subcritical Galton-Watson process; this analysis will also show that the probability that $C_m=\es$ for some $m\ge1$ tends to $1$ as $n\to\infty$. Indeed, consider the subcritical Galton-Watson process whose offspring distribution is Binomial with $2^d$ trials and success probability $2^{-d-1}$.
Put $k=k_0(n)=\lfloor \log(2\sqrt{d}n/t_0(1/2))/\log 2\rfloor$.
 It follows from~\eqref{lowCovProp} and the independence property of the Poisson point process that \tes a coupling between $X^{(n)}$ and the Galton-Watson process \st for every $k\le k_0$ we have $T_k\ge\sum_{i=0}^k\#A^{(1)}_i$. Here $T_k$ denotes the total progeny of the Galton-Watson process up to the $k$th step. \Ip,
$$\max\{\P(C_{k_0}\ne\es),\P(\#B>2^dT_{k_0})\}\le \P(T_{k_0-1}\ne T_{k_0})\le 2^{-k_0}\le t_0(1/2)/(\sqrt{d}n),$$
where the second inequality uses a well-known result for Galton-Watson processes, see e.g.~\cite[Theorem 5.1]{harBra}.
\end{proof}

\section{Using a multi-level topology for redundancy elimination}
\label{multiSec}
In the present section, we provide a comparison between the scale-free Gilbert graph $G(X^{(n)})$ and the thinned scale-free Gilbert graph $G^\p(X^{(n)})$, where some redundant edges are removed. First, we prove Theorem~\ref{redElThm}.

\begin{proof}[Proof of Theorem~\ref{redElThm}]
In the following we fix $t_0=t_0(1/2)$.  
We first show that we can neglect contributions $b_{1}(n)$ to $\E D^{\p,(\alpha)}_{\out,n}$ coming from points $(\eta,t)\in X^{(n)}$ \st $t>\sqrt{|\eta|}$. Indeed,
\begin{align*}
b_1(n)&\le\int_0^\infty\int_{\BT_r(o)}|\eta|^\alpha\P(R>\sqrt{|\eta|})\d\eta\P_R(\d r)\\
&=\int_{\T_n}|\eta|^\alpha\P(R>|\eta|)\P(R>\sqrt{|\eta|})\d\eta\\
&=\int_{\T_n\setminus \BT_{t^2_0}(o)}|\eta|^\alpha\P(R>|\eta|)\P(R>\sqrt{|\eta|})\d\eta+\int_{\BT_{t^2_0}(o)}|\eta|^\alpha\P(R>|\eta|)\P(R>\sqrt{|\eta|})\d\eta.
\end{align*}
Since the second expression remains bounded as $n\to\infty$, it suffices to consider the first. Then, 
\begin{align*}
&\int_{\T_n\setminus \BT_{t^2_0}(o)}|\eta|^\alpha\P(R>|\eta|)\P(R>\sqrt{|\eta|})\d\eta\le 2\beta^2 \int_{\T_n\setminus \BT_{t^2_0}(o)} |\eta|^{\alpha-3d/2}\d\eta \\
&\qquad=2\beta^2\int_{\T_n\setminus \BT_{n/2}(o)}|\eta|^{\alpha-3d/2}\d\eta+2\beta^2\int_{\BT_{n/2}(o)\setminus\BT_{t^2_0}(o)}|\eta|^{\alpha-3d/2}\d\eta,
\end{align*}
and we consider the two summands separately. Clearly, the first is in $O(n^{\alpha-d/2})$. For the second we obtain that
\begin{align*}
\int_{\BT_{n/2}(o)\setminus\BT_{t^2_0}(o)}|\eta|^{\alpha-3d/2}\d\eta&=d\kappa_d\int_{t_0^2}^{n/2}u^{\alpha-d/2-1}\d u,
\end{align*}
which is in $O(1)$ if $\alpha=0$ and in $O(n^{\max\{0,\alpha-d/4\}})$ if $\alpha>0$.
\Fm, we can clearly neglect the contributions $b_{2}(n)$ to $\E D^{\p,(\alpha)}_{\out,n}$ coming from points $(\eta,t)\in X^{(n)}$ with $|\eta|\le\min\{e^{12}, t_0^2\}$. 

It remains to obtain bounds for the contributions that are covered by neither $b_1(n)$ nor $b_2(n)$. For any $\gamma>0$ and $\xi\in [-n/2,n/2]^d$ we put
$$S_{\gamma,\eta}=\{\zeta\in\R^d: \sqrt{|\eta|}\le|\zeta-\eta|\le |\eta|\text{ and }\angle (\zeta-\eta,-\eta) \in[-\gamma,\gamma]\};$$
see Figure~\ref{angleFig} for an illustration of $S_{\gamma,\eta}$. If $\gamma$ is sufficiently small, then for every $n\ge1$ and $\eta\in [-n/2,n/2]^d$ the sector $S_{\gamma,\eta}$ is contained in $[-n/2,n/2]^d$. In the following, we fix any such value $\gamma_0$  and put $S_{\eta}=S_{\gamma_0,\eta}$.
It will also be convenient to denote by $S^\p_{\eta}=\partial B_1(o)\cap (|\eta|^{-1}(S_{\eta}-\eta))$ the intersection of the unit sphere with a shifted and scaled copy of $S_{\eta}$. Finally, we denote by $\sigma_0=\nu_{d-1}(S^\p_\eta)$ the surface area of $S^\p_\eta$, a quantity which is independent of $\eta$.

\begin{figure}[!htpb]
\centering
\begin{tikzpicture}[scale=1.85]
\draw (0,0) rectangle (2,2); 
\filldraw[black!40!white] (1.0,1.0) arc (-125:-145:1.2207cm) --(1.7,2);
\filldraw[black!40!white] (1.0,1.0) arc (-125:-105:1.2207cm) --(1.7,2);
\filldraw[white](1.35,1.5) arc (-125:-104:0.61cm)--(1.7,2);
\filldraw[white](1.35,1.5) arc (-125:-146:0.61cm)--(1.7,2);
\fill (1.7,2) circle (1pt);
\draw[dashed] (1,1)--(1.7,2);
\draw (1.35,1.5) arc (-125:-145:0.61cm)--(1.7,2);
\draw (1.35,1.5)--(1.7,2);

\coordinate[label=90:\small{$\eta$}] (u) at (1.7,2);
\coordinate[label=90:\small{$\gamma$}] (u) at (1.3,1.51);
\end{tikzpicture}
\caption{Construction of the set $S_{\gamma,\eta}$ (shaded)}
\label{angleFig}
\end{figure}
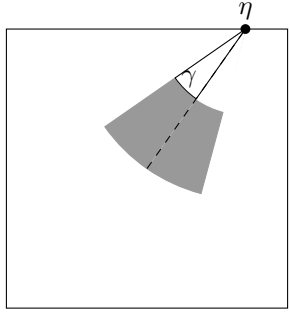

If $(\eta,t)$ is an out-neighbor of $(o,R^*)$ in $G^\p(\{o,R^*\}\cup X^{(n)})$ \st $t\le \sqrt{|\eta|}$, then $X^{(n)}\cap A_{\eta,t}=\es$, where 
$$A_{\eta,t}=\{(\zeta,w)\in S_{\eta}\times[0,R^*):\,w>|\zeta-\eta|\}.$$

\Fm, we note that conditioned on $R^*=r$ the number of points of $X^{(n)}$ contained in $A_{\eta,t}$ is a Poissonian random variable with mean
\begin{align*}
\int_{S_{\eta}} \P(R\in (|\zeta-\eta|,r)) \d \zeta&=\sigma_0\int_{\sqrt{|\eta|}}^{|\eta|}u^{d-1}\P(R\in(u,r))\d u.
\end{align*}
If $n\ge1$ is sufficiently large and $|\eta|\ge t_0^2$, then the right-hand side is at least
$$ \beta\sigma_0\int_{\sqrt{|\eta|}}^{|\eta|} \tfrac{1}{2}u^{-1}-\tfrac{3}{2}u^{d-1}r^{-d} \d u\ge \beta\sigma_0(\tfrac{1}{4}\log |\eta| -\tfrac{3}{2})\ge \tfrac{1}{8}\beta\sigma_0\log |\eta|.$$
Putting $b_3(n)=\E D^{\p,(\alpha)}_{\out,n}-b_1(n)-b_2(n)$, we therefore obtain that
\begin{align*}
b_3(n)&\le \int_0^\infty\int_{\BT_{r}(o)\setminus \BT_{t^2_0}(o)}|\eta|^\alpha\exp(-\tfrac{1}{8}\beta\sigma_0\log |\eta|) \d\eta \P_R(\d r)\\
&= \int_{\T_n\setminus \BT_{t^2_0}(o)}|\eta|^{\alpha-\beta\sigma_0/8} \P(R>|\eta|)\d\eta\\
&\le 2\beta\int_{\T_n\setminus \BT_{n/2}(o)} |\eta|^{\alpha-d-\beta\sigma_0/8}\d\eta+2\beta\int_{ \BT_{n/2}(o)\setminus \BT_{t^2_0}(o)} |\eta|^{\alpha-d-\beta\sigma_0/8}\d\eta\\
&\le 2\beta\int_{\T_n\setminus \BT_{n/2}(o)} |\eta|^{\alpha-d-\beta\sigma_0/8}\d\eta+2\beta d\kappa_d\int_{t^2_0}^{n/2} u^{\alpha-\beta\sigma_0/8-1}\d u.
\end{align*}
Observing that the last line is in $O(1)$ if $\alpha=0$ and in $O(n^{\max\{0,\alpha-\beta\sigma_0/16\}})$ if $\alpha>0$ completes the proof.
 \end{proof}

When passing from $G(X^{(n)})$ to $G^\p(X^{(n)})$ only redundant edges are removed, in the sense that connected components remain unchanged.
\begin{proposition}
\label{conProp}
With probability $1$ the graphs $G(X^{(n)})$ and $G^\p(X^{(n)})$ have the same connected components.
\end{proposition}
\begin{proof}
It suffices to show that if $x=(\xi,r), y=(\eta,t)\in X^{(n)}$ are connected by an edge in $G(X^{(n)})$, then $x$ and $y$ are contained in the same connected component of $G^\p(X^{(n)})$. Suppose that this was false. Then, choose a counter-example with the property that $|r-t|$ is minimal. Without loss of generality we may assume $r\ge t$. Since $x=(\xi,r)$ and $y=(\eta,t)$ are not connected by an edge in $G^\p(X^{(n)})$, \tes a point $z=(\zeta,w)\in X^{(n)}\cap(\BT_{r}(\xi)\times(t,r))$ \st $\eta\in\BT_{w}(\zeta)$. By the minimality of the counter-example we see that both $y$ and $z$ as well as $z$ and $x$ are contained in the same connected component of $G^\p(X^{(n)})$. Therefore also $x$ and $y$ are contained in the same connected component of $G^\p(X^{(n)})$, yielding a contradiction to the initial assumption.
\end{proof}

Next, we show that when moving from $G(X^{(n)})$ to $G^\p(X^{(n)})$ chemical distances increase at most by a logarithmic factor in the size of the torus. To achieve this goal, we make use of a variant of the descending chains concept introduced in~\cite{lilypond}. To be more precise, let $x_1=(\xi_1,r_1),\ldots,x_m=(\xi_m,r_m)$ be \st $r_1>r_2>\cdots>r_m$. Then $x_1,\ldots,x_m$ are said to form a \emph{toroidal descending chain} if $\xi_{i+1}\in \BT_{r_{i}}(\xi_{i})$ for every $i\in\{1,\ldots,m-1\}$. In the following result, we show that there is a close relationship between the existence of short connections in $G^\p(X^{(n)})$ and the absence of long toroidal descending chains.
\begin{lemma}
\label{detDescLem}
Let $x=(\xi,r)\in X^{(n)}$ and $y=(\eta,t)\in X^{(n)}\cap(\BT_{r}(\xi)\times(0,r))$. If $m\ge1$ is \st the chemical distance between $x$ and $y$ in $G^\p(X^{(n)})$ is larger than $m$, then \tes a toroidal descending chain starting from $x$ and consisting of more than $m$ points.
\end{lemma}
\begin{proof}
We proceed similarly to~\cite[Lemma 10]{aldLin}.
Inductively, we construct points $\big\{x^{(i)}_j\big\}_{0\le j\le i }=\big\{(\xi^{(i)}_j,r^{(i)}_j)\big\}_{0\le j\le i}$ \st for every $i\in\{1,\ldots,m\}$ we have
\begin{enumerate}
\item $r^{(i)}_0>\cdots>r^{(i)}_i$,
\item $\xi^{(i)}_{j+1}\in \BT_{r^{(i)}_{j}}(\xi_j^{(i)})$ \fa $j\in\{1,\ldots,i-1\}$, and
\item \tes $j\in\{1,\ldots,i-1\}$ \st $x^{(i)}_j$ and $x^{(i)}_{j+1}$ are not connected by an edge in $G^\p(X^{(n)})$.
\end{enumerate}
Indeed, for the induction start we just choose $x^{(1)}_0=x$ and $x^{(1)}_1=y$. Next, suppose that $i\le m$ and that we have constructed points $\big\{x^{(i-1)}_{j}\big\}_{0\le j\le i}$ in $X^{(n)}$ satisfying properties (1)-(3). Choose $j_0\in\{0,\ldots,i-2\}$ \st  $x^{(i-1)}_{j_0}$ and $x^{(i-1)}_{j_0+1}$ are not connected by an edge in $G^\p(X^{(n)})$. By definition of $G^\p(X^{(n)})$ \tes $x^\p=\(\xi^\p,r^\p\)\in X^{(n)}\cap\Big(\BT_{r^{(i-1)}_{j_0}}(\xi^{(i-1)}_{j_0})\times \(r^{(i-1)}_{j_0+1},r^{(i-1)}_{j_0}\)\Big)$ \st $\xi^{(i-1)}_{j_0+1}\in \BT_{r^\p}(\xi^\p)$. We put
$$x^{(i)}_{j}=
\begin{cases}
x^{(i-1)}_j&\text{if $j\le j_0$},\\
x^\p&\text{if $j=j_0+1$},\\
x^{(i-1)}_{j-1}&\text{if $j>j_0+1$}.
\end{cases}$$
Clearly, properties (i) and (ii) are satisfied. Since a violation of property (iii) would imply that the chemical distance between $x$ and $y$ in $G^\p(X^{(n)})$ is at most $m$, this completes the proof.
\end{proof}

For the applicability of Lemma~\ref{detDescLem}, it is important to show that long toroidal descending chains can occur only with small probability.
\begin{lemma}
\label{probDescChainLem}
There exists a constant $c_1>0$ \st whp $X^{(n)}$ does not contain a toroidal descending chain consisting of more than $c_1\log n$ elements.
\end{lemma}
\begin{proof}
We distinguish several cases depending on the radii occurring in the chain. First, note that the probability that \te distinct $(\xi,r),(\eta,t)\in X^{(n)}$ \st $\min\{r,t\}>n^{3/4}$ is at most $n^{2d}\P(R>n^{3/4})^2$, which tends to $0$ as $n\to\infty$. Hence, it suffices to consider toroidal descending chains with radii bounded above by $n^{3/4}$. Next, we note that the expected number of toroidal descending chains consisting of $m+1$ steps and where all radii are bounded above by $t_0=t_0(1/2)$ is at most 
${n^d}(\kappa_dt_0^d)^m/{m!}.$
Using Stirling's formula, we see that when putting $m=\log n$, this expression tends to $0$ as $n\to\infty$. Finally, the expected number of toroidal descending chains consisting of $m\ge1$ steps, and where all radii are contained in $(t_0,n^{3/4})$ is at most
\begin{align*}
&\int_{\T_n}\int_{t_0}^{n^{3/4}}\int_{\BT_{r_1}(\xi_1)}\cdots\int_{t_0}^{r_{m-1}}\int_{\BT_{r_m}(\xi_m)}1\d\xi_{m+1}\P_R(\d r_m)\cdots\d\xi_2\P_R(\d r_1)\d\xi_1\\
&\qquad= n^d\kappa_d^m \int_{t_0}^{n^{3/4}}\cdots\int_{t_0}^{n^{3/4}}1_{r_1>\cdots>r_m}r_1^d\cdots r_m^d\P_R(\d r_m)\cdots\P_R(\d r_1)\\
&\qquad=\frac{n^d\kappa_d^m}{m!} \Big(\int_{t_0}^{n^{3/4}}r^d \P_R(\d r)\Big)^m.
\end{align*}
In order to derive an upper bound for the latter expression, we need to investigate $\E R^d 1_{(t_0,n^{3/4})}(R)$. We compute
\begin{align*}
\E R^d 1_{(t_0,n^{3/4})}(R)&=\int_0^{t_0^d}\P(R>t_0)\d t+\int_{t_0^d}^{n^{3d/4}}\P(R^d>t)\d t\\
&\le t_0^d+2\beta \int_{1}^{n^{3d/4}}t^{-1}\d t\\
&\le 4d\beta\log n,
\end{align*}
provided that $n\ge1$ is sufficiently large. \Ip, the expected number of toroidal chains consisting of $m$ steps, and where all radii are contained in $(t_0,n^{3/4})$ is at most 
$$n^d (4d\kappa_d\beta\log n)^m/m!,$$
and Stirling's formula implies that when putting $m=c\log n$ this expression is at most 
$$\exp(d\log n +c(\log n)(\log(4de\kappa_d\beta\log n)-\log (c \log n))).$$
Finally, the latter expression tends to $0$ as $n\to\infty$, if $c>0$ is sufficiently large.
\end{proof}

Combining Lemmas~\ref{detDescLem} and~\ref{probDescChainLem} the proof of Theorem~\ref{intMedProp} is now immediate.
\begin{proof}[Proof of Theorem~\ref{intMedProp}]
Lemma~\ref{detDescLem} shows that if $x,y\in X^{(n)}$ are connected by an edge in $G(X^{(n)})$ and the chemical distance between $x$ and $y$ in $G^\p(X^{(n)})$ is at least $c_1\log n$, then $X^{(n)}$ contains a toroidal descending chain consisting of at least $c_1\log n$ elements. Lemma~\ref{probDescChainLem} shows that the complements of the latter events occur whp.
\end{proof}

\label{lastpage}
\subsection*{Acknowledgments}
The author thanks I.~Norros for introducing him to the papers~\cite{aby1,koski}, which marked the starting point for the present research. Moreover, the author thanks R.~van der Hofstad for pointing out the continuum variant of the scale-free percolation model,~\cite{wuth1,wuth2}.

\end{document}